\documentclass[12pt,leqno,twoside]{article} 
\usepackage{bezier,amsmath,amssymb,stmaryrd,amsthm,upgreek}
\usepackage[colorlinks, linkcolor=black, citecolor=black, urlcolor=black, hyperfootnotes=true]{hyperref}                
\usepackage{multicol}
\usepackage{mathtools}
\usepackage{dsfont}
\usepackage{multirow}
\usepackage{alltt}
\usepackage{longtable}
\input xy
\xyoption{all}

\renewcommand{\baselinestretch}{1.1}
\newtheorem{Satz}{Satz} 
\newtheorem{Example}[Satz]{Example}
\newtheorem{Lemma}[Satz]{Lemma}
\newtheorem{Theorem}[Satz]{Theorem}
\newtheorem{Corollary}[Satz]{Corollary}

\newtheorem{Definition}[Satz]{Definition}

\newtheorem{Proposition}[Satz]{Proposition}
\newtheorem{Remark}[Satz]{Remark}

\newcommand{\ncm}{\textstyle}
\newcommand{\bfcite}[2]{{\bf\cite[\rm #2]{#1}}}
\newcommand{\bfcit}[1]{{\bf\cite{#1}}}
\newcommand{\minus}{\smallsetminus}

\newcommand{\ih}[1]{\overset{\raisebox{-1mm}{$\sscm\vee$}}{#1}}

\newcommand{\s}[1]{\tilde{#1}} 
\newcommand{\Res}{\operatorname{Res}}

\newcommand{\ru}[1]{\rule[#1mm]{0mm}{0mm}}

\newcommand{\sscm}{\scriptscriptstyle}
\newcommand{\scr}{\scriptsize}
\newcommand{\EE}{\mathrm{E}}
\newcommand{\Z}{\mathds{Z}}
\newcommand{\Q}{\mathds{Q}}

\newcommand{\aufgl}[1]{\stackrel{\mbox{\scr #1}}{=}}

\newcommand{\val}{\operatorname{v}\!}

\renewcommand{\phi}{\varphi}
\renewcommand{\epsilon}{\varepsilon}
\newcommand{\vsp}[1]{\vspace*{#1mm}}
\newcommand{\hsp}[1]{\hspace*{#1mm}}

\newcommand{\Sylv}{\operatorname{Sylv}}

\newcommand{\bq}{\begin{quote}\begin{footnotesize}}
\newcommand{\eq}{

                 \end{footnotesize}\end{quote}} 
\newcommand{\ti}{\times}
\newcommand{\ubrt}{\underbracket[0.5pt]} 

\newcommand{\kommentar}[1]{}

\DeclareMathOperator{\GL}{GL}

\begin{document}


\title{A resultant for Hensel's Lemma}
\author{Juliane Dei\ss ler}
\maketitle

\begin{footnotesize}

\bq
Let $R$ be a complete discrete valuation ring with maximal ideal generated by $\pi$. Let $f(X)\in R[X]$ be a monic polynomial with nonzero discriminant $\Delta(f)$. Let 
$s \ge \val_\pi(\Delta(f)) + 1$. Suppose given a factorisation
of $f(X)$ in $(R/\pi^s R)[X]$ into several factors, not necessarily coprime in $(R/\pi R)[X]$. We lift it uniquely to a factorisation of $f(X)$ in $R[X]$. This generalises the 
Hensel-Rychl\'ik Lemma, which
covers the case of two factors. We work directly with lifts of factorisations into several factors and avoid having to iterate factorisations into two factors.
For this purpose we define a resultant for several polynomials in one variable as determinant of a generalised Sylvester matrix.
\eq 

\renewcommand{\thefootnote}{\fnsymbol{footnote}}
\renewcommand{\thefootnote}{\arabic{footnote}}

\renewcommand{\baselinestretch}{0.7}%
\parskip0.0ex%
\tableofcontents%
\parskip1.2ex%
\renewcommand{\baselinestretch}{1.0}%
\end{footnotesize}%

\setcounter{section}{-1}

\section{Introduction}

In this introduction, by a polynomial we understand a monic polynomial.

Let $R$ be a complete discrete valuation ring and $\pi \in R$ a generator of its maximal ideal. 

\subsection{Resultant of several polynomials}
\label{ss_Res}

Classically, one defines the resultant of two polynomials in $R[X]$ as the determinant of their Sylvester matrix; cf. {\sc van der Waerden} \bfcite{vdW60}{\S 34}.

Here, for polynomials $g_{(1)}(X), \dots, g_{(n)}(X) \in R[X]$, where $n\ge 1$, the resultant $\Res(g_{(1)},\dots,g_{(n)})$ 
is defined to be the determinant of the generalised Sylvester matrix $\Sylv(g_{(1)},\dots,g_{(n)})$. This matrix contains the coefficients of the products 
$\prod_{j\in[1,n]\minus\{k\}} g_{(j)}(X)$, where $1\le k\le n$, ordered in a similar way as in the classical Sylvester matrix; cf. Definition~\ref{Def1}. 

We obtain the product formula
$
\Res(g_{(1)} \,,\, \dots \,,\, g_{(n)}) = \prod_{1\le\, k\, <\, \ell\,\le n} \Res(g_{(k)} \,,\, g_{(\ell)}) \, ,
$
expressing our resultant as a product of classical resultants. Since in our application to Hensel lifting, 
the generalised Sylvester matrix $\Sylv(g_{(1)},\dots,g_{(n)})$ itself will appear, it would not have been possible to work only with this product formula. 

Suppose given a polynomial $f(X)\in R[X]$ with discriminant $\Delta(f)\ne 0$. Let \mbox{$s\ge\val_\pi(\Delta(f)) + 1$.} 
As shall be explained in \S\ref{SecGenCase} below, we will start a Hensel lifting process with a factorisation
$
f(X) \;\equiv_{\pi^s}\; \prod_{k\in[1,n]} g_{(k)}(X)\,,
$
that is, with a factorisation of precision $s$, into factors $g_{(k)}(X)$, which are not necessarily coprime modulo $\pi$.
In this context, the product formula for the resultant yields the inequality
$2\, \val_{\pi}(\Res(g_{(1)} \,,\, \dots \,,\, g_{(n)})) \,\leq\, \val_{\pi}(\Delta(f))\,$; cf.\ Proposition~\ref{bewertung}. Technically speaking, this inequality is the reason
why we may start the process with said precision. In summary, the resultant machinery will supply our Hensel lifting process with a starting precision that depends only on 
the polynomial $f(X)$ to be factorised.

\subsection{Applications to Hensel's Lemma}
\label{App_Hensel's_Lemma}

\subsubsection{General case}
\label{SecGenCase}

Hensel's Lemma in the classical sense~\bfcite{Ko97}{4.4.2} has, in rudimentary form, already been known to {\sc Gauss} \bfcite{Ga76}{\S 374}; cf.\ \bfcite{Fr07}{\S 3.6}.
{\sc Hensel} developed a more sophisticated version~\bfcite{He04}{\S 4, p.\ 80}, known today as Hensel-Rychl\'ik Lemma.  
We generalise in Theorem~\ref{KochHensel} the Hensel-Rychl\'ik Lemma from the case of two factors to the case of an arbitrary number of factors.

Let $f(X)\in R[X]$ such that $\Delta(f)\ne 0$. Let $s\ge\val_\pi(\Delta(f)) + 1$. Choose a factorisation 
\[
f(X) \;\equiv_{\pi^s}\; \prod_{k \in [1,n]} g_{(k)}(X)\; , 
\]
where $n\ge 1$, i.e.\ a factorisation 
modulo $\pi^s$ into factors $g_{(k)}(X)\in R[X]$ of degree $\ge 1$, which are not necessarily coprime modulo $\pi$. Write $t := \val_\pi(\Res(g_{(1)}\,,\,\dots\,,\,g_{(n)}))$. 
Then there exist unique polynomials 
$\ih{g}_{(1)}(X)$, \dots\ , $\ih{g}_{(n)}(X)$ in $R[X]$ congruent, respectively, to \linebreak 
$g_{(1)}(X)$, \dots, $g_{(n)}(X)$ modulo $\pi^{s-t}$ such that 
\[
f(X) \;=\; \prod_{k \in [1,n]} \ih{g}_{(k)}(X)\; .
\] 

So to find a factorisation of a polynomial $f(X)$ in $R[X]$, we start with a factorisation of precision~$s$, satisfying a lower bound depending only on $f(X)$, and lift it to a 
factorisation of $f(X)$ in $R[X]$. Here, to lift means to replace the old factors $g_{(k)}(X)$, factoring modulo $\pi^s$, by new factors $\ih{g}_{(k)}(X)$, factoring exactly, such 
that the new factors are congruent to the old factors modulo~$\pi^{s - t}$. Concerning the connection from new to old, $s$ is the precision ``one might hope for'' -- but from that we have
to subtract $t$, which thus plays the role of a ``potential defect''; this $t$, in turn, is bounded above depending only on $f(X)$, viz.\ $t\le\val_\pi(\Delta(f))/2$. 
Such a defect actually occurs in examples, cf.\ \S~\ref{SecEx}.

The inductive step of the proof of Theorem~\ref{KochHensel} is contained in Lemma~\ref{KochHenselStepPrae}.
The arguments for that step I have learnt from {\sc Koch} \bfcite{Ko97}{4.4.3, 4.4.4, 4.4.5}.

In Example~\ref{ExComparison} we suppose given a factorisation $f(X) \equiv_{\pi^s} g_{(1)}(X) \cdot g_{(2)}(X) \cdot g_{(3)}(X)$ into three factors, in order 
to compare the result of a single application of 
Lemma~\ref{KochHenselStepPrae} to three factors with the result of two subsequent applications of Lemma~\ref{KochHenselStepPrae} to two factors. We determine that both 
methods are essentially equally good.

\subsubsection{Particular case $f(X) \equiv_\pi X^M$}
\label{SecPartCase}

In \S~\ref{SecSp} we investigate our generalisation of the Hensel-Rychl\'ik Lemma in a particular case. Here, slightly better bounds than in the general case hold. Namely, the bound for the
starting precision and the bound for the defect can be lowered somewhat, compared to the general case.

Let $f(X)$ be a polynomial in $R[X]$ with $\deg f =: M$ and $f(X) \equiv_\pi X^M$. Let $n\ge 1$ and $g_{(1)}(X),\, \dots,\, g_{(n)}(X)\in R[X]$ of 
degree $\ge 1$, ordered such that $\deg g_{(1)} \le \dots \le \deg g_{(n)}$\,. Again, we write $t := \val_\pi(\Res(g_{(1)}\,,\,\dots\,,\,g_{(n)})$. Moreover, we write 
$t' := t - \sum_{j\in [1,n-1]} \big((n-j)(\deg g_{(j)}) - 1\big)$. Now, suppose that $f(X) \equiv_{\pi^s} \prod_{k \in [1,n]} g_{(k)}(X)$ for some $s \ge t + t' + 1$. 
Then there exist unique polynomials $\ih{g}_{(1)}(X),\,\dots,\,\ih{g}_{(n)}(X)\,\in\,R[X]$ congruent, respectively, to $g_{(1)}(X),\,\dots,\,g_{(n)}(X)$ modulo $\pi^{s-t'}$ 
such that $f(X) = \prod_{k \in [1,n]} \ih{g}_{(k)}(X)$. Cf.\ Theorem~\ref{KochHenselSp}.

The proof of Theorem~\ref{KochHenselSp} is similar to the respective proof in the general case. We refrained from attempting 
to produce an assertion that covers both the general Theorem~\ref{KochHensel} and the more particular Theorem~\ref{KochHenselSp}, for it probably 
would have obscured the proof of Theorem~\ref{KochHensel}.

Examples~\ref{Ex_u37}~and~\ref{Ex_u41} show that $t' < t$ may occur. 

The inductive step for the proof of Theorem~\ref{KochHenselSp} is contained in Lemma~\ref{KochHenselStepPraeSp}.
In Example~\ref{ExComparisonSp}, we suppose given a factorisation $f(X) \equiv_{\pi^s} g_{(1)}(X) \cdot g_{(2)}(X) \cdot g_{(3)}(X)$ into three factors, 
ordered such that $\deg g_{(1)} \le \deg g_{(2)} \le \deg g_{(3)}\,$, in order to compare the result of a single application of Lemma~\ref{KochHenselStepPraeSp} to three factors with the 
result of two subsequent applications of 
Lemma~\ref{KochHenselStepPraeSp} to two factors. Under the present hypothesis $f(X) \equiv_\pi X^M$, we determine that the former method yields a somewhat more 
precise result than the latter method.

\subsection{Acknowledgements}
\label{SecAck}

To illustrate the theory we consider in \S~\ref{SecEx} some polynomials with cofficients in $\Z_p$ for a prime number $p$ using the computer algebra system 
{\sc Magma} \bfcit{Magma}.

I thank the referee of an earlier version for arguments that considerably simplified \S~\ref{SecRes} and for pointing out the reference~\bfcit{FPR02}; cf.\ Remark~\ref{ReferentenMethode}.
 
\subsection{Notations}
\setlength{\parskip}{0mm}
\begin{itemize}\setlength{\itemsep}{0mm}
\item Given $a,\,b\,\in\,\Z$, we denote by $[a,b] := \{z\in\Z\, :\, a\le z\le b\}\subseteq\Z$ the integral interval.
\item Given an integral domain $R$, a prime element $\pi\in R$ with $\pi \ne 0$ and $x\in R\minus\{ 0\}$, we denote 
$\val_\pi(x) := \max\{\,i\in\Z_{\ge 0} \,:\, \text{$\pi^i$ divides $x$}\,\}$. Moreover, $\val_\pi(0) := +\infty$.
\item We denote by $\EE_m$ the unit matrix of size $m\ti m$.
\item Given a commutative ring $R$ and elements $x,\,y,\,u\,\in\,R$, we write $x\equiv_u y$ for $x - y \in uR$.
\item For the zero polynomial $0$, we put $\deg 0 := -\infty$.
\end{itemize}
\setlength{\parskip}{0.9mm}

\section{Resultants}
\label{SecRes}

Let $R$ be an integral domain. Let $\pi \ne 0$ be a prime element of $R$.

\subsection{A lemma}
\label{SecLem}

\begin{Definition}
\label{DefB}
Suppose given $z\ge 0$. Suppose given $s\ge 0$. Suppose given $u(X) = \sum_{i\in [0,s]} u_i X^i$. Let
\[
B_{z,\,z+s}(u) \;\;\; :=\;\;\;
\left(
\begin{array}{cccccccccc}
 u_0 & u_1 & \cdots & \cdots & \cdots & u_s    &        &        &        &                 \\
     & u_0 & u_1    & \cdots & \cdots & \cdots & u_s    &        &        &                 \\
     &     & u_0    & u_1    & \cdots & \cdots & \cdots & u_s    &        &                 \\
     &     &        & \ddots & \ddots &        &        &        & \ddots &                 \\
     &     &        &        & u_0    & u_1    & \cdots & \cdots & \cdots & u_s             \\
\end{array}
\right)
\;\;\; \in\;\;\; R^{z\ti (z+s)}\; .
\]
\end{Definition}

\begin{Lemma}
\label{LemBfBgBfg}
Suppose given $z\ge 0$. Suppose given $s,\,t\,\ge\, 0$. Suppose given $u(X) = \sum_{i\in [0,s]} u_i X^i$. Suppose given $v(X) = \sum_{i\in [0,t]} v_i X^i$. 
Then 
\[
B_{z,\,z+s}(u)\cdot B_{z+s,\,z+s+t}(v) \; =\; B_{z,\,z+s+t}(uv)\; . 
\]
\end{Lemma}

\begin{proof}
Write $u(X) = \sum_{i\ge 0} u_i X^i$ with $u_i = 0$ for $i\ge s+1$ and $v(X) = \sum_{i\ge 0} v_i X^i$ with $v_i = 0$ for $i\ge t+1$. 
Then $u(X)\cdot v(X) = \sum_{j\ge 0} \left(\sum_{i\in [0,j]} u_i v_{j-i}\right) X^j$.

Suppose given $p\in [1,z]$ and $q\in [1,z+s+t]$. If $p > q$, then the entry at position $(p,q)$ is zero in the left hand side and in the right hand side matrix. If $p\le q$, then
the entry at position $(p,q)$ of $B_{z,\,z+s}(u) B_{z+s,\,z+s+t}(v)$ equals $\sum_{i\in [0,q-p]} u_i v_{q-p-i}\,$, which equals the entry at position $(p,q)$ of $B_{z,\,z+s+t}(uv)$.
\end{proof}

\subsection{A resultant}
\label{SecARes}

Let $n\in\Z_{\ge 1}\,$. Suppose given monic polynomials
$g_{(k)} = g_{(k)}(X) = \sum\limits_{i\in[0,m_{(k)}]} g_{(k)i}X^{i} \in R[X]$,
where $m_{(k)} := \deg g_{(k)} \ge 1$, for $k\in [1,n]$. Denote $M := \smash{\sum\limits_{j\in[1,n]}} m_{(j)}\,$. Denote $M_{(k)} := M - m_{(k)}$ and
\[
\smash{\prod\limits_{j\in[1,n]\minus\{k\}}} g_{(j)}(X) \;=:\; \sum\limits_{i\in[0,M_{(k)}]} a_{(k)i}X^{i}
\]
for $k\in [1,n]$.

Let $K$ be the field of fractions of $R$. Let $L$ be a splitting field for $\prod_{k\in [1,n]} g_{(k)}(X) \,\in\, K[X]$. 

\begin{Definition}
\label{Def1}\rm
Let
\[
 \Sylv(g_{(1)},\dots,g_{(n)}) \;\;:=\;\;
\left(
\begin{array}{ccccccc}
 a_{(1)0} & \cdots & \cdots   & \cdots & a_{(1)M_{(1)}}\hsp{-3} &        &         \\
          & \ddots &          &        &                        & \ddots &         \\
          &        & a_{(1)0} &\cdots  &\cdots                  & \cdots & a_{(1)M_{(1)}} \hspace{2mm}\smash{\raisebox{7mm}{$\left.\ru{9.5}\right\}\text{\scr$m_{(1)}$ rows}$}}\hspace*{-24.5mm} \\
 a_{(2)0} & \cdots & \cdots   & \cdots & a_{(2)M_{(2)}}\hsp{-3} &        &         \\
          & \ddots &          &        &                        & \ddots &         \\
          &        & a_{(2)0} & \cdots & \cdots                 & \cdots & a_{(2)M_{(2)}} \hspace{3mm}\smash{\raisebox{7mm}{$\left.\ru{9.5}\right\}\text{\scr$m_{(2)}$ rows}$}}\hspace*{-23.5mm}\\
 \vdots   &        & \vdots   &        & \vdots                 &        & \vdots  \\
 \vdots   &        & \vdots   &        & \vdots                 &        & \vdots  \\
a_{(n)0}  & \cdots & \cdots   & \cdots & a_{(n)M_{(n)}}\hsp{-3} &        &         \\
          & \ddots &          &        &                        & \ddots &         \\
          &        & a_{(n)0} & \cdots & \cdots                 & \cdots & a_{(n)M_{(n)}} \hspace{4mm}\smash{\raisebox{7mm}{$\left.\ru{9.5}\right\}\text{\scr$m_{(n)}$ rows}$}}\hspace*{-22mm}\\
\end{array}\right) \hsp{20}\in\;\; R^{M\ti M}\; .
\]
Let
\[
\Res(g_{(1)},\dots,g_{(n)}) \;\;:=\;\; \det \Sylv(g_{(1)},\dots,g_{(n)}) \;\in\; R
\]
be the {\it resultant} of $g_{(1)}(X) \,,\dots,\, g_{(n)}(X)$.  
\end{Definition}

Note that
\[ 
\Sylv(g_{(1)},\dots,g_{(n)})
\;=\;
\left(
\begin{array}{c}
B_{m_{(1)},\,M}\!\left(\prod\limits_{j\in[1,n]\minus\{1\}} g_{(j)}(X)\right)\ru{-6} \\\hline
B_{m_{(2)},\,M}\!\left(\prod\limits_{j\in[1,n]\minus\{2\}} g_{(j)}(X)\right)\ru{-6}\ru{8} \\\hline
\vdots\ru{6} \\
\vdots\ru{-4} \\\hline
B_{m_{(n)},\,M}\!\left(\prod\limits_{j\in[1,n]\minus\{n\}} g_{(j)}(X)\right)\ru{8} \\
\end{array}
\right)\;.
\]

In the case $n = 1$, we have $\Res(g_{(1)}(X)) = 1$, for $\Sylv(g_{(1)}) = \EE_{m_{(1)}}\,$.

In the case $n = 2$, our resultant coincides with the resultant found in the literature, e.g.\ in \linebreak 
\bfcite{Co82}{\S 7.4, (4)}, because it is obtained by a row and column reordering that leaves the determinant unchanged, for we need 
$m_{(1)}m_{(2)} + \lfloor\frac{m_{(1)} + m_{(2)}}{2}\rfloor + \lfloor\frac{m_{(1)}}{2}\rfloor + \lfloor\frac{m_{(2)}}{2}\rfloor \equiv_2 0$ transpositions.

\begin{Lemma}[cf.\ \bfcite{Co82}{\S 7.4, Th.~2.(iv), Cor.}]
\label{LemCohn}
Suppose given $\,u(X),\,v(X),\,w(X)\,\in\, R[X]$ monic. 
\begin{itemize}

\item[\rm (1)] Choose a field $C$ containing $R$ as a subring such that $v(X) = \prod_{i\in [1,s]} (X - \alpha_i)$ and\linebreak 
$w(X) = \prod_{j\in [1,t]} (X - \beta_j)$ in $C[X]\,$, where $s := \deg u$
and $t := \deg v$. 

Then
$
\;\Res(v,w) \;=\; \prod\limits_{(i,j)\,\in\, [1,s]\ti [1,t]} (\alpha_i - \beta_j)\; .
$
\item[\rm (2)] We have 
$
\;\Res(uv,w) = \Res(u,w)\cdot\Res(v,w)\; .
$ 
\end{itemize}
\end{Lemma}

\begin{Lemma}
\label{LemReferee}
\hsp{1}We have $\;\;\Res(g_{(1)} \,,\, \dots \,,\, g_{(n)}) \;\;=\;\; \prod\limits_{1\le\, k\, <\, \ell\,\le n} \Res(g_{(k)} \,,\, g_{(\ell)})\;\;$.
\end{Lemma}

\begin{proof}
We proceed by induction on $n\ge 1\,$; basing it at $n = 1$, where both sides equal $1$.

Suppose given $n\ge 2$. Suppose the assertion known for $n - 1$. Then
\[
\begin{array}{l}
\Sylv(g_{(1)},\dots,g_{(n)})
\;=\;
\left(
\begin{array}{c}
B_{m_{(1)},\,M}\!\left(\prod\limits_{j\in[1,n]\minus\{1\}} g_{(j)}(X)\right)\ru{-6}       \\\hline
B_{m_{(2)},\,M}\!\left(\prod\limits_{j\in[1,n]\minus\{2\}} g_{(j)}(X)\right)\ru{-6}\ru{8} \\\hline
\vdots\ru{6}\ru{-3}                                                                       \\\hline
B_{m_{(n)},\,M}\!\left(\prod\limits_{j\in[1,n]\minus\{n\}} g_{(j)}(X)\right)\ru{8}        \\
\end{array}
\right) \vsp{2} \\
\;\aufgl{L.\,\ref{LemBfBgBfg}}\;
\left(
\begin{array}{c|c}
B_{m_{(1)},\,M_{(n)}}\!\left(\prod\limits_{j\in[1,n-1]\minus\{1\}} g_{(j)}(X)\right)\ru{-6}           & \\\hline
B_{m_{(2)},\,M_{(n)}}\!\left(\prod\limits_{j\in[1,n-1]\minus\{2\}} g_{(j)}(X)\right)\ru{-6}\ru{8}     & \\\hline
\vdots\ru{-3}                                                                                         & \\\hline
B_{m_{(n-1)},\,M_{(n)}}\!\left(\prod\limits_{j\in[1,n-1]\minus\{n-1\}} g_{(j)}(X)\right)\ru{-6}\ru{8} & \\\hline
                                                                                                      & \EE_{m_{(n)}} \ru{8}\ru{-5}\\
\end{array}
\right) \cdot
\left(
\begin{array}{c}
\\
\\
\\
B_{M_{(n)},\,M}(g_{(n)}(X)) \\
\\
\\
\\\hline
B_{m_{(n)},\,M}\!\left(\prod\limits_{j\in[1,n-1]} g_{(j)}(X)\right)\ru{8}
\end{array}
\right) \; .
\end{array}
\]
Taking determinants, this yields
\[
\begin{array}{rcl}
\Res(g_{(1)} \,,\, \dots \,,\, g_{(n)}) 
& =                             & \Res(g_{(1)} \,,\, \dots \,,\, g_{(n-1)})\cdot\Res(\prod_{j\in[1,n-1]} g_{(j)}\,,\,g_{(n)})                                                      \vsp{1} \\
& \aufgl{L.\ \ref{LemCohn}.(2)} & \Res(g_{(1)} \,,\, \dots \,,\, g_{(n-1)})\cdot \prod_{j\in[1,n-1]} \Res(g_{(j)}\,,\,g_{(n)})                                                     \vsp{1} \\
& \aufgl{induction}             & \Big(\prod\limits_{1\le\, k\, <\, \ell\,\le n-1} \Res(g_{(k)} \,,\, g_{(\ell)})\Big)\cdot\Big(\prod_{j\in[1,n-1]} \Res(g_{(j)}\,,\,g_{(n)})\Big) \vsp{1} \\
& =                             & \prod\limits_{1\le\, k\, <\, \ell\,\le n} \Res(g_{(k)} \,,\, g_{(\ell)})\; .                                                                             \\
\end{array}
\]
\end{proof}

Lemma~\ref{LemReferee} together with Lemma~\ref{LemCohn}.(1) gives the

\begin{Corollary}
 \label{res=prod}
Write $g_{(k)}(X) \,=:\, \prod\limits_{i\in[1,m_{(k)}]} (X-\gamma_{(k)i})\;$ in $L[X]$ for $k\in [1,n]$. 

We have
\[
\Res(g_{(1)},\dots,g_{(n)}) \;\; =\;\; 
\prod_{1\le\, k\, <\, \ell\,\le n} \;\;\prod_{(i,j)\,\in\, [1,m_{(k)}]\ti [1,m_{(\ell)}]} (\gamma_{(k)i} \,-\, \gamma_{(\ell)j}) \; .
\]
\end{Corollary}

\subsection{Resultants and discriminants}
\label{ResDisc}

Denote by $\Delta(f)$ the discriminant of a polynomial $f(X)\in R[X]$.

\begin{Corollary}
 \label{disk=diskmalres}
We have
$
 \Delta(g_{(1)} \cdot\, \dots\,\cdot g_{(n)}) \,=\, \big(\prod\limits_{k\in [1,n]}\Delta(g_{(k)})\big) \,\cdot\, 
 \Res(g_{(1)},\, \dots \,, g_{(n)})^{2} \; .
$
\end{Corollary}

\begin{proof}
Note that $\Delta(g_{(k)}) = \prod\limits_{1 \leq i<j \leq m_{(k)}} (\gamma_{(k)i}-\gamma_{(k)j})^{2}\,$ for $k \in [1,n]$; 
cf.\ \bfcite{vdW60}{\S 33}. So
\[
\hsp{10}
 \begin{array}{cl}
                             & \hsp{-20}\Delta(g_{(1)}\cdot\, \dots \,\cdot g_{(n)}) \\
=                            & (\prod\limits_{k\in [1,n]} \prod\limits_{\;1 \leq i<j \leq m_{(k)}} 
                               (\gamma_{(k)i} - \gamma_{(k)j})^{2})\, \cdot\, 
                               (\prod\limits_{1\le\, k\, <\, \ell\,\le n} \;\;\prod\limits_{(i,j)\,\in\, [1,m_{(k)}]\ti [1,m_{(\ell)}]} (\gamma_{(k)i} \,-\, \gamma_{(\ell)j})^2)  \\
\aufgl{C\,\ref{res=prod}} & (\prod\limits_{k\in [1,n]} \Delta(g_{(k)})) \,\cdot\, 
                               \Res(g_{(1)} \,,\, \dots \,,\, g_{(n)})^{2} \; . \vsp{-5}\\
\end{array}
\]
\end{proof}

\begin{Remark}
 \label{diskfkongrdisksf}
Let $r\in R$. Suppose given monic polynomials $f(X),\, \s{f}(X) \,\in\, R[X]$ such that $f(X) \,\equiv_r\, \s{f}(X)$. Then 
$
 \Delta(f) \,\equiv_r\, \Delta(\s{f}) .
$
\end{Remark}

\begin{Proposition}
\label{bewertung}
Let $f(X)\in R[X]$ be a monic polynomial with $\Delta(f)\ne 0$. Suppose that we have 
$
f(X) \;\equiv_{\pi\Delta(f)}\; \prod\limits_{k\in[1,n]} g_{(k)}(X)\;.
$
Then $\Res(g_{(1)} \,,\, \dots \,,\, g_{(n)})\ne 0$ and \hsp{1.5}
\[
2\, \val_{\pi}(\Res(g_{(1)} \,,\, \dots \,,\, g_{(n)})) \;\leq\; \val_{\pi}(\Delta(f))\; .
\]
\end{Proposition}

\begin{proof}
We have 
$
\Delta(f)\;\overset{\text{R~\ref{diskfkongrdisksf}}}{\equiv}_{\!\pi\Delta(f)}\; \Delta(\prod\limits_{k\in[1,n]} g_{(k)}) \; \aufgl{C~\ref{disk=diskmalres}}\; 
\big(\prod\limits_{k\in[1,n]} \Delta(g_{(k)})\big) \,\cdot\, \Res(g_{(1)} \,,\, \dots \,,\, g_{(n)})^{2}\; .
$
\end{proof}

\begin{Remark}
\label{bewertungRes}
Let $r\in R$. Let $\s{g}_{(1)}(X) \,,\, \dots \,,\, \s{g}_{(n)}(X) \, \in\, R[X]$ be monic polynomials
such that 
$
g_{(k)}(X) \;\equiv_r\; \s{g}_{(k)}(X)
$
for $k\in [1,n]$. Then\hsp{1.5}
$
\Res(g_{(1)} \,,\, \dots \,,\, g_{(n)}) \;\equiv_r\; \Res(\s{g}_{(1)} \,,\, \dots \,,\, \s{g}_{(n)})\; .
$
\end{Remark}

\begin{proof} We may assume that $r$ is not a unit in $R$. Then $\deg g_{(k)} = \deg \s{g}_{(k)}$ for $k\in [1,n]$. Hence
$
\Sylv(g_{(1)},\dots,g_{(n)})\equiv_r \Sylv(\s{g}_{(1)},\dots,\s{g}_{(n)})\; ;
$
cf.\ Definition~\ref{Def1}. Taking determinants, we get
$
\Res(g_{(1)} \,,\, \dots \,,\, g_{(n)}) \equiv_r \Res(\s{g}_{(1)} \,,\, \dots \,,\, \s{g}_{(n)})\; .
$
\end{proof}

\section{Hensel}
\label{SecKH}

Let $R$ be a discrete valuation ring. Let $\pi\in R$ be a generator of the maximal ideal of~$R$. 

\subsection{Linear Algebra tools}
\label{SecLA}

Suppose given $k\ge 1$. Suppose given $A\in R^{k\times k}$ such that $\det(A)\ne 0$.
Let $\pi^{e_1}$, \dots, $\pi^{e_k}$ be the elementary divisors of $A$, ordered such that $0\le e_1\le e_2\le\dots\le e_k$\,. 
Write $e := e_1 +\cdots + e_k = \val_\pi(\det(A))$. Choose $S,\,T \,\in\, \GL_k(R)$ such that $SAT = \text{diag}(\pi^{e_1},\dots,\pi^{e_k}) =: D$.
Suppose given $d_i\in\Z_{\ge 0}$ for $i\in [1,k]$ such that $d_1\ge d_2\ge\dots\ge d_k$. Write $e' := e - (d_2 + \cdots + d_k)$.

\begin{Remark}
\label{RemMinors}
Suppose that for every $i\in [1,k]$, the element $\pi^{d_i}$ divides each entry in column number $i$ of $A$. 
Then $0\le e_k \le e'$.
\end{Remark}

\begin{proof}
Each $(k-1)\ti (k-1)$-minor of $A$ is divisible by $\pi^{d_k + \dots + d_2}$. So $\pi^{d_k + \dots + d_2}$ divides their greatest common divisor $\pi^{e - e_k}$.
\end{proof}

\begin{Lemma}
\label{LinearAlgebra1}
\indent\vsp{-1}
\begin{itemize}
\item[\rm (1)] Suppose given $y\in \pi^{e_k} R^{1\ti k}$. Then there exists $x\in R^{1\ti k}$ such that $xA = y$.
\item[\rm (2)] Suppose given $y\in \pi^e R^{1\ti k}$. Then there exists $x\in R^{1\ti k}$ such that $xA = y$.
\item[\rm (3)] Suppose that for every $i\in [1,k]$, the element $\pi^{d_i}$ divides each entry in column number $i$ of $A$. 
Suppose given $y\in \pi^{e'} R^{1\ti k}$. Then there exists $x\in R^{1\ti k}$ such that $xA = y$.
\end{itemize}
\end{Lemma}

\begin{proof}

Ad~(1). Write
$
yT \;=:\; (\pi^{e_k} z_1 \,,\, \dots \,,\, \pi^{e_k} z_k ) 
$,
where $z_i \in R$ for $i \in [1,k]$. Let
$
x \;:=\;  (\pi^{e_k - e_1} z_1 \,,\, \dots \,,\, \pi^{e_k - e_k} z_k )S \;\in\; R^{1\ti k} \;.
$
So $xA = xS^{-1}DT^{-1} = (\pi^{e_k - e_1} z_1 \,,\, \dots \,,\, \pi^{e_k - e_k} z_k )DT^{-1} = yTT^{-1} = y$.

Ad~(3). By Remark~\ref{RemMinors} we have $e' \geq e_k\,$, so that the assertion follows with (1).
\end{proof}

\begin{Lemma}
\label{LinearAlgebra2}
\indent\vsp{-1}
\begin{itemize}
\item[\rm (1)] Suppose given $u\ge e_k$ and $x\in R^{1\ti k}$ such that $xA\in R^{1\ti k} \pi^u$. Then $x\in R^{1\ti k} \pi^{u - e_k}$.
\item[\rm (2)] Suppose given $u\ge e$ and $x\in R^{1\ti k}$ such that $xA\in R^{1\ti k} \pi^u$. Then $x\in R^{1\ti k} \pi^{u - e}$.
\item[\rm (3)] Suppose that for 
every $i\in [1,k]$, the element $\pi^{d_i}$ divides each entry in column number $i$ of $A$. 
Suppose given $u\ge e'$ and $x\in R^{1\ti k}$ such that $xA\in R^{1\ti k} \pi^u$. Then $x\in R^{1\ti k} \pi^{u - e'}$.
\end{itemize}
\end{Lemma}

\begin{proof}
Ad~(1). We have $xA \;=\; xS^{-1}DT^{-1} \in R^{1\ti k} \pi^u$, whence $xS^{-1}D \in R^{1\ti k} \pi^u$.   
Denote\linebreak
$
xS^{-1} =: (z_1 \,,\, \dots \,,\, z_k) \in R^{1\ti k}
$.
So 
$
xS^{-1}D = (\pi^{e_1}z_1 \,,\, \dots \,,\, \pi^{e_k}z_k) 
$.
Hence $z_i \in R\pi^{u - e_i} \subseteq R\pi^{u - e_k}$ for $i \in[1,k]$. So
$
xS^{-1} = (z_1 \,,\, \dots \,,\, z_k) \in  R^{1\ti k}\pi^{u - e_k}  
$.
Hence $x \in  R^{1\ti k}\pi^{u - e_k}S  =  R^{1\ti k}\pi^{u - e_k}$.

Ad~(3). By Remark~\ref{RemMinors} we have $e' \geq e_k\,$, so that the assertion follows with (1). 
\end{proof}

\subsection{General case}
\label{SecGen}

Let $f(X)\in R[X]$ be a monic polynomial such that $\Delta(f)\ne 0$. Write $M := \deg f$. 

Let $n\ge 1$. Let $g_{(1)}(X)\,,\, \dots \,,\, g_{(n)}(X) \in R[X]$ be monic polynomials of degree $\ge 1$. Denote $t := \val_{\pi}(\Res(g_{(1)}\,,\,\dots\,,\, g_{(n)}))$.
Write $m_{(k)} := \deg g_{(k)}$ and $M_{(k)} := M - m_{(k)}$ for $k\in [1,n]$.

Let $s\geq \val_\pi(\Delta(f)) + 1$. Suppose that
$
f(X) \equiv_{\pi^{s}} \prod\limits_{k \in [1,n]} g_{(k)}(X)\; .
$

Note that $t$ is finite and that $s\geq 2t + 1$; cf.\ Proposition~\ref{bewertung}. 

(Actually, we could also suppose only the condition $s\geq 2t + 1$. However, $2t+1$ depends on the factors $g_{(1)}(X)\,,\, \dots \,,\, g_{(n)}(X)$, whereas $\val_\pi(\Delta(f)) + 1$ 
depends only on $f(X)$.)

\begin{Lemma}[cf.\ \bfcite{He04}{p.\ 81}]
\label{KochHenselStepPrae}\indent\vsp{-1}
\begin{itemize}
\item[\rm (1)] There exist monic polynomials $\s{g}_{(1)}(X) \,,\, \dots \,,\, \s{g}_{(n)}(X) \,\in\, R[X]$
such that \\
$
\s{g}_{(k)}(X) \;\equiv_{\pi^{s-t}}\; g_{(k)}(X) 
$
for $k \in [1,n]$ and
$
f(X) \;\equiv_{\pi^{2(s - t)}}\; \prod\limits_{k \in [1,n]} \s{g}_{(k)}(X)\; .
$

We call such a tuple $(\s{g}_{(k)}(X))_k$ of polynomials an {\em admissible lift} of $(g_{(k)}(X))_k$ with respect to $s$.
We have
$
\,\val_\pi(\Res(\s{g}_{(1)}\,,\,\dots\,,\, \s{g}_{(n)})) \;=\; t
$
for any admissible lift $(\s{g}_{(k)}(X))_k$ of $(g_{(k)}(X))_k$ with respect to $s$.

\item[\rm (2)] Suppose given $r\in [0,s - 2t]$. 
Suppose given monic polynomials $\s{g}_{(1)}(X) \,,\, \dots \,,\, \s{g}_{(n)}(X)$, $\s{h}_{(1)}(X) \,,\, \dots \,,\, \s{h}_{(n)}(X) \in R[X]$
such that
$
\s{g}_{(k)}(X) \equiv_{\pi^{s-t}} g_{(k)}(X) 
$
and
$
\s{h}_{(k)}(X) \equiv_{\pi^{s-t}} g_{(k)}(X) 
$
for $k \in [1,n]$, and
$
\prod\limits_{k \in [1,n]} \s{g}_{(k)}(X)\equiv_{\pi^{2(s - t) - r}} \prod\limits_{k \in [1,n]} \s{h}_{(k)}(X).
$
Then 
$
\s{g}_{(k)}(X) \,\equiv_{\pi^{2s - 3t - r}}\, \s{h}_{(k)}(X)
$
for $k\in [1,n]$. In particular, considering the case $r = 0$, two admissible lifts with respect to $s$ as in {\rm (1)} are mutually congruent modulo 
$\pi^{2s - 3t}R[X]$.
\end{itemize}
\end{Lemma}

In the following proof, we shall use the notation of \S~\ref{SecRes}. The arguments I have learnt 
from {\sc Koch} \bfcite{Ko97}{Satz 4.4.3, Hilfssatz 4.4.4, Hilfssatz 4.4.5}.

\begin{proof}
{\it Ad~{\rm (1).}} {\it Existence of admissible lift.}

We make the ansatz
$
\s{g}_{(k)}(X) \,=\, g_{(k)}(X)+\pi^{s-t}u_{(k)}(X) 
$
for $k \in [1,n]$ with $u_{(k)}(X) \,\in\, R[X]$ and $\deg u_{(k)} < \deg g_{(k)} = m_{(k)}$ for $k \in [1,n]$. Thus we require that
\[
\begin{array}{rll}
f(X) 
& \overset{!}{\equiv}_{\pi^{2(s - t)}} & \prod\limits_{k \in [1,n]} \s{g}_{(k)}(X)                                                                    
\; =\;                                     \prod\limits_{k \in [1,n]} (g_{(k)}(X)+\pi^{s-t}u_{(k)}(X))                                                \\
& \equiv_{\pi^{2(s - t)}}              & \prod\limits_{k \in [1,n]} g_{(k)}(X) 
                                          + \pi^{s-t}\sum\limits_{k\in [1,n]}u_{(k)}(X)\cdot\prod\limits_{\ell\in [1,n]\minus\{k\}} g_{(\ell)}(X) \; .\\
\end{array}
\]

Let $b(X):= \pi^{t-s}(f(X)-\prod\limits_{k \in [1,n]}g_{(k)}(X))$. Since $f(X) \equiv_{\pi^{s}} \prod\limits_{k \in [1,n]}g_{(k)}(X)$, we get 
$b(X) \equiv_{\pi^{t}} 0$.
So our requirement reads 
$
b(X) \;\overset{!}\equiv_{\pi^{s - t}} \sum\limits_{k \in [1,n]} u_{(k)}(X) \,\cdot\, 
\prod\limits_{\ell \in [1,n]\minus\{k\}} g_{(\ell)}(X)\; .
$
So it suffices to find $u_{(k)}(X) \in R[X]$ for $k \in [1,n]$ as above such that
$
 b(X) \;\overset{!}{=} \sum\limits_{k \in [1,n]} u_{(k)}(X) \,\cdot\, \prod\limits_{\ell \in [1,n]\minus\{k\}} g_{(\ell)}(X) \; .
$

Writing 
$
b(X)                                                  =:  \sum\limits_{i\geq 0} \beta_i X^i\,
$,
$
\prod\limits_{\ell\in[1,n]\minus\{k\}} g_{(\ell)}(X)  =:  \sum\limits_{i \geq 0} a_{(k)i}X^{i}\, 
$,
$
u_{(k)}(X)                                            =:  \sum\limits_{i \geq 0} u_{(k)i}X^{i}\, 
$
for $k\in [1,n]$, where $\beta_{i}\,,\, a_{(k)i}\,,\, u_{(k)i}  \,\in\, R$ for $i\ge 0$, a comparison of coefficients shows that it suffices to 
find 
\[
U \,:=\, (\ubrt{\ru{-4}u_{(1)\,0} \, \dots \, u_{(1)\,m_{(1)}-1}}\; \ubrt{\ru{-4}u_{(2)\,0} \, \dots \, u_{(2)\,m_{(2)}-1}} \;\;\,\dots\,\;\; 
          \ubrt{\ru{-4}u_{(n)\,0} \,\dots\, u_{(n)\,m_{(n)}-1}})\,\in\, R^{1\ti M}
\]
such that
\[
\hsp{-12}
U \cdot
\underbrace{\left(
\begin{array}{ccccccl}
a_{(1)0}  & \cdots & \cdots   & \cdots & a_{(1)M_{(1)}}\hsp{-3} &        &         \\
          & \ddots &          &        &                        & \ddots &         \\
          &        & a_{(1)0} &\cdots  &\cdots                  & \cdots & a_{(1)M_{(1)}} \hsp{4.5}\smash{\raisebox{7mm}{$\left.\ru{9.5}\right\}\text{\scr$m_{(1)}$ rows}$}}\hspace*{-24.5mm} \\
 a_{(2)0} & \cdots & \cdots   & \cdots & a_{(2)M_{(2)}}\hsp{-3} &        &         \\
          & \ddots &          &        &                        & \ddots &         \\
          &        & a_{(2)0} & \cdots & \cdots                 & \cdots & a_{(2)M_{(2)}} \hsp{4.5}\smash{\raisebox{7mm}{$\left.\ru{9.5}\right\}\text{\scr$m_{(2)}$ rows}$}}\hspace*{-23.5mm}\\
 \vdots   &        & \vdots   &        & \vdots                 &        & \vdots  \\
 \vdots   &        & \vdots   &        & \vdots                 &        & \vdots  \\
a_{(n)0}  & \cdots & \cdots   & \cdots & a_{(n)M_{(n)}}\hsp{-3} &        &         \\
          & \ddots &          &        &                        & \ddots &         \\
          &        & a_{(n)0} & \cdots & \cdots                 & \cdots & a_{(n)M_{(n)}} \hsp{4}\smash{\raisebox{7mm}{$\left.\ru{9.5}\right\}\text{\scr$m_{(n)}$ rows}$}}\hspace*{-22mm}\\
\end{array}\right)}_{\ncm =\; \Sylv(g_{(1)} \,,\, \dots \,,\,g_{(n)})}
\hspace{20mm} \aufgl{!} \,\,\, (\beta_0 \,\dots\, \beta_{M-1}).
\]
Note that $(\beta_0\,\dots\,\beta_{M-1}) \in \pi^{t} R^{1\ti M}$ since $b(X) \equiv_{\pi^{t}} 0$. So $U$ exists as
required by Lemma~\ref{LinearAlgebra1}.(2).

{\it Valuation of resultant.} Since $\s{g}_{(k)}(X) \equiv_{\pi^{s - t}} g_{(k)}(X)$ for $k\in [1,n]$, Remark~\ref{bewertungRes} implies that
$
\Res(\s{g}_{(1)}\,,\,\dots\,,\, \s{g}_{(n)})\equiv_{\pi^{s - t}}\Res(g_{(1)}\,,\,\dots\,,\, g_{(n)}).
$
Since $s - t \ge t + 1 = \val_\pi(\Res(g_{(1)}\,,\,\dots\,,\, g_{(n)})) + 1$, this implies
$
\,\val_\pi(\Res(\s{g}_{(1)}\,,\,\dots\,,\, \s{g}_{(n)}))\,=\,\val_\pi(\Res(g_{(1)}\,,\,\dots\,,\, g_{(n)}))\, =\, t\, .
$

{\it Ad~{\rm (2).}} Writing
$
\s{g}_{(k)}(X) =: g_{(k)}(X) + \pi^{s - t} u_{(k)}(X)     
$
and
$
\s{h}_{(k)}(X) =: g_{(k)}(X) + \pi^{s - t} v_{(k)}(X)
$
for $k\in [1,n]$, where $u_{(k)}(X),\,v_{(k)}(X)\,\in\, R[X]$, we obtain $\deg u_{(k)}(X) < \deg g_{(k)}(X) = m_{(k)}\,$, 
since $\s{g}_{(k)}(X)$ and $g_{(k)}(X)$ are monic polynomials of the same degree; likewise, we obtain \mbox{$\deg v_{(k)}(X) < m_{(k)}\,$}.

We have to show that $u_{(k)}(X)\overset{!}{\equiv}_{\pi^{s-2t-r}} v_{(k)}(X)$ for $k \in [1,n]$.
We have
\[
\begin{array}{c}
\prod\limits_{k \in [1,n]} g_{(k)}(X) + \pi^{s-t}\sum\limits_{k\in [1,n]}u_{(k)}(X)\cdot\prod\limits_{\ell\in [1,n]\minus\{k\}} g_{(\ell)}(X) 
\;\equiv_{\pi^{2(s - t)}}\;\prod\limits_{k \in [1,n]} (g_{(k)}(X)+\pi^{s-t}u_{(k)}(X))                                                      \\
\;=\; \prod\limits_{k \in [1,n]} \s{g}_{(k)}(X) \;\equiv_{\pi^{2(s - t) - r}}\; \prod\limits_{k \in [1,n]} \s{h}_{(k)}(X)                     \\
\;=\; \prod\limits_{k \in [1,n]} (g_{(k)}(X)+\pi^{s-t}v_{(k)}(X)) 
\;\equiv_{\pi^{2(s - t)}}\; \prod\limits_{k \in [1,n]} g_{(k)}(X) 
                                + \pi^{s-t}\sum\limits_{k\in [1,n]}v_{(k)}(X)\cdot\prod\limits_{\ell\in [1,n]\minus\{k\}} g_{(\ell)}(X) \; .\\
\end{array}
\]
The difference yields
$
{\textstyle
\sum\limits_{k\in [1,n]} (u_{(k)}(X) - v_{(k)}(X))\cdot\prod\limits_{\ell\in [1,n]\minus\{k\}} g_{(\ell)}(X) \,\equiv_{\pi^{s - t-r}}\, 0\, .
}
$

Writing
$
w_{(k)}(X) \,:=\, u_{(k)}(X) - v_{(k)}(X)
$
for $k\in [1,n]$, this reads\hypertarget{ast.0}%
\[
\sum\limits_{k\in [1,n]} w_{(k)}(X)\cdot\prod\limits_{\ell\in [1,n]\minus\{k\}} g_{(\ell)}(X) \;\equiv_{\pi^{s - t - r}}\; 0 \; .
\leqno (\ast)
\]
Writing 
$
w_{(k)}(X) \;=:\; \sum\limits_{i\geq 0} w_{(k)i}X^{i}
$
for $k\in [1,n]$, and
\[
W \,:=\, (\ubrt{\ru{-4}w_{(1)\,0} \, \dots \, w_{(1)\,m_{(1)}-1}}\; \ubrt{\ru{-4}w_{(2)\,0} \, \dots \, w_{(2)\,m_{(2)}-1}} \;\;\,\dots\,\;\; 
          \ubrt{\ru{-4}w_{(n)\,0} \,\dots\, w_{(n)\,m_{(n)}-1}})\,\in\, R^{1\ti M}\; ,
\]
we have to show that $W\overset{!}{\in}\pi^{s - 2t - r} R^{1\ti M}$. From \hyperlink{ast.0}{$(\ast)$}, we obtain
$
W \cdot \Sylv(g_{(1)} \,,\, \dots \,,\,g_{(n)}) \,\in\, \pi^{s - t - r}R^{1\ti M}\; .
$
Note that $s - t - r\ge t= \val_\pi(\det \Sylv(g_{(1)} \,,\, \dots \,,\,g_{(n)}))$. 
So we can infer by Lemma~\ref{LinearAlgebra2}.(2) that $W\in \pi^{s - 2t - r}R^{1\ti M}$.
\end{proof}

\begin{Remark}
\label{ReferentenMethode}\rm
In \bfcite{FPR02}{App.~B}, the same method is used as in Lemma~\ref{KochHenselStepPrae}.(1). Since in \mbox{\bfcite{FPR02}{App.~B}}, the discriminant of $f(X)$ is not invoked, 
the necessary starting precision depends on the factorisation chosen there; namely, it is assumed, using our notation and context, that for some $\tau\ge 0$, the congruence
$U\cdot \Sylv(g_{(1)} \,,\, \dots \,,\, g_{(n)}) \equiv_{\pi^{\tau+1}} (\pi^\tau,0,\dots,0)$ is solvable for some $U\in R^{1\ti M}$ in order to be able to use a factorisation 
$f(X)\equiv_{\pi^{2\tau +1}} \prod_{k\in [1,n]} g_{(k)}(X)$. This condition is satisfied if $\Res(g_{(1)} \,,\, \dots \,,\, g_{(n)})$ divides $\pi^\tau$, and so, using Proposition~\ref{bewertung},
it is satisfied if $\Delta(f)$ divides $\pi^{2\tau}$.
\end{Remark}

\begin{Theorem}
\label{KochHensel}
Suppose $R$ to be a complete discrete valuation ring. Recall that $f(X)\in R[X]$ is a monic polynomial with $\Delta(f)\ne 0$, that $s \ge \val_\pi(\Delta(f)) + 1$, 
that $f(X) \equiv_{\pi^{s}} \smash{\prod\limits_{k \in [1,n]}} g_{(k)}(X)$ and that \linebreak $t = \val_{\pi}(\Res(g_{(1)}\,,\,\dots\,,\, g_{(n)}))$.

Then there exist unique monic polynomials \mbox{$\ih{g}_{(1)}(X) \,,\, \dots \,,\, \ih{g}_{(n)}(X) \,\in\, R[X]$} such that 
$
\;\ih{g}_{(k)}(X) \,\equiv_{\pi^{s-t}}\, g_{(k)}(X)\;
$
for $k \in [1,n]$ and
$
\;f(X) \,=\, \prod\limits_{k \in [1,n]} \ih{g}_{(k)}(X)\; .
$
\end{Theorem}

\begin{proof}
 
{\it Existence.} Since $R$ is complete, it suffices to show that there exist monic polynomials
$\s{g}_{(1)}(X) \,,\, \dots \,,\, \s{g}_{(n)}(X) \,\in\, R[X]$ such that 
$
f(X) \,\equiv_{\pi^{s+1}}\, \prod\limits_{k \in [1,n]} \s{g}_{(k)}(X)
$
and
$
\s{g}_{(k)}(X) \,\equiv_{\pi^{s-t}}\, g_{(k)}(X)\;
$
for $k \in [1,n]$, and
$
\val_{\pi}(\Res(\s{g}_{(1)}\,,\,\dots\,,\, \s{g}_{(n)})) = t. 
$
This follows by Lemma~\ref{KochHenselStepPrae}.(1) as $2(s - t) \ge s+1$.

{\it Uniqueness.} Given 
$\ih{g}_{(1)}(X) , \dots , \ih{g}_{(n)}(X) ,\ih{h}_{(1)}(X) , \dots , \ih{h}_{(n)}(X) \in R[X]$, all monic, such that
\mbox{$
\ih{g}_{(k)}(X) \equiv_{\pi^{s-t}} g_{(k)}(X) \equiv_{\pi^{s-t}} \ih{h}_{(k)}(X)
$}
for \mbox{$k \in [1,n]$} and
\smash{$
\prod\limits_{k \in [1,n]} \ih{g}_{(k)}(X) = f(X) = \prod\limits_{k \in [1,n]} \ih{h}_{(k)}(X),
$}
we have to show that $\ih{g}_{(k)}(X)\overset{!}{=} \ih{h}_{(k)}(X)$ for $k \in [1,n]$.

Note that $\val_{\pi}(\Res(\ih{g}_{(1)}\,,\,\dots\,,\, \ih{g}_{(n)})) = t = \val_{\pi}(\Res(\ih{h}_{(1)}\,,\,\dots\,,\, \ih{h}_{(n)}))$
by Lemma~\ref{KochHenselStepPrae}.(1).

Let $s_1 := s$. Both $(\ih{h}_{(k)}(X))_k$ and $(\ih{g}_{(k)}(X))_k$ are admissible lifts of $(\ih{g}_{(k)}(X))_k$ with respect to $s_1\,$ 
in the sense of Lemma~\ref{KochHenselStepPrae}.(1), 
since $\ih{h}_{(k)}(X) \equiv_{\pi^{s_1 - t}} \ih{g}_{(k)}(X)$ for $k\in [1,n]$ and the
other required congruences are verified using equalities. So Lemma~\ref{KochHenselStepPrae}.(2) yields
$
\ih{h}_{(k)}(X) \equiv_{\pi^{2(s_1 - t) - t}} \ih{g}_{(k)}(X) 
$
for $k\in [1,n]$.

Let $s_2 := 2(s_1 - t)$. Note that $s_2 = s_1 + (s_1 - 2t) > s_1\,$. Both $(\ih{h}_{(k)}(X))_k$ and $(\ih{g}_{(k)}(X))_k$ are 
admissible lifts of $(\ih{g}_{(k)}(X))_k$ with respect to $s_2\,$, since $\ih{h}_{(k)}(X) \equiv_{\pi^{s_2 - t}} \ih{g}_{(k)}(X)$ for $k\in [1,n]$.
So Lemma~\ref{KochHenselStepPrae}.(2) yields
$
\ih{h}_{(k)}(X) \equiv_{\pi^{2(s_2 - t) - t}} \ih{g}_{(k)}(X) 
$
for $k\in [1,n]$.

Let $s_3 := 2(s_2 - t)$. Note that $s_3 = s_2 + (s_2 - 2t) > s_2 + (s_1 - 2t) > s_2\,$. Continue as above.

This yields a strictly increasing sequence $(s_\ell)_{\ell\,\ge\, 1}$ of integers such that 
$
\ih{h}_{(k)}(X) \,\equiv_{\pi^{s_\ell - t}}\, \ih{g}_{(k)}(X) 
$
for $k\in [1,n]$ and $\ell\ge 1$. Hence
$
\ih{h}_{(k)}(X) \,=\, \ih{g}_{(k)}(X) 
$
for $k\in [1,n]$.
\end{proof}

\begin{Remark}
\label{RemHensel}\rm
The case $n = 2$ of Theorem~\ref{KochHensel}, i.e.\ the case of a factorisation of $f(X)$ into two factors $g_{(1)}(X)$ and $g_{(2)}(X)$ modulo $\pi^s$, 
is due to {\sc Hensel}; cf.\ \bfcite{He04}{p.\ 80, 81}. 

Translated to our notation, he starts right away with $s > t$. He writes in the statement on \mbox{\bfcite{He04}{p.\ 80, l.\ 8}} that 
$\ih{g}_{(1)}(X)$ and $\ih{g}_{(2)}(X)$ are ``N\"aherungswerte'' of $g_{(1)}(X)$ and $g_{(2)}(X)$. In the proof, on \bfcite{He04}{p.\ 81, l.\ 7}, he 
makes this precise and shows that actually $\ih{g}_{(1)}(X)\equiv_{\pi^{s-t}} g_{(1)}(X)$ and $\ih{g}_{(2)}(X)\equiv_{\pi^{s-t}} g_{(2)}(X)$.
\end{Remark}

\begin{Example}
\label{ExComparison}\rm
Suppose that $n = 3$. Write
$
t_0 := \val_\pi(\Res(g_{(2)} \,,\, g_{(3)}))
$,
$
t_1 :=  \val_\pi(\Res(g_{(1)} \,,\, g_{(2)}g_{(3)})) 
$.
Corollary~\ref{res=prod} gives $\Res(g_{(1)} \,,\, g_{(2)} \,,\, g_{(3)}) = \Res(g_{(1)} \,,\, g_{(2)}g_{(3)}) \cdot \Res(g_{(2)} \,,\, g_{(3)})$,
whence $t = t_1 + t_0\,$. In particular, $t_0$ and $t_1$ are finite.

We can apply Lemma~\ref{KochHenselStepPrae}.(1) to  $f(X) \equiv_{\pi^{s}} g_{(1)}(X)\cdot g_{(2)}(X)\cdot g_{(3)}(X)$ 
to obtain monic polynomials $\s{g}_{(1)}(X) \,,\, \s{g}_{(2)}(X)\,,\, \s{g}_{(3)}(X)\,\in\, R[X]$ such 
that\hypertarget{i}{}
\[
\s{g}_{(k)}(X) \;\equiv_{\pi^{s-t}} \;g_{(k)}(X) \hsp{2} \text{for $k \in [1,3]$} \; ,\hsp{3}         
 f(X)          \;\equiv_{\pi^{2(s - t)}} \; \s{g}_{(1)}(X)\cdot\s{g}_{(2)}(X)\cdot\s{g}_{(3)}(X) \, .  
\leqno \text{(i)}
\]
We can also apply Lemma~\ref{KochHenselStepPrae}.(1) first to the factorisation $f(X) \equiv_{\pi^{s}} g_{(1)}(X)\cdot (g_{(2)}(X)\cdot g_{(3)}(X))$ and
then to the resulting factorisation of the second factor into $g_{(2)}(X)\cdot g_{(3)}(X)$ modulo a certain power of $\pi$. 

We have $s \geq 2t+1 \geq 2t_1+1$. So Lemma~\ref{KochHenselStepPrae}.(1) gives monic polynomials $\s{h}_{(1)}(X) \,,\, \s{h}_{(2)}(X) \,\in\,R[X]$ such that
\[
 \s{h}_{(1)}(X) \; \equiv_{\pi^{s-t_1}} \; g_{(1)}(X) \; , \;\;
 \s{h}_{(2)}(X) \; \equiv_{\pi^{s-t_1}} \; g_{(2)}(X)\cdot g_{(3)}(X)  \;,\;\;    
 f(X)           \; \equiv_{\pi^{2(s - t_1)}}  \; \s{h}_{(1)}(X) \cdot \s{h}_{(2)}(X) \;. 
\]
We have $s - t_1\geq 2t+1 - t_1 = t_1 + 2t_0 + 1 \geq 2t_0 + 1$. So  Lemma~\ref{KochHenselStepPrae}.(1) gives
monic polynomials $\s{\s{g}}_{(2)}(X)\,,\,\s{\s{g}}_{(3)}(X)\,\in\,R[X]$ such that 
\[
 \begin{array}{c}
 \s{\s{g}}_{(2)}(X)  \;\equiv_{\pi^{(s - t_1) - t_0}} \; g_{(2)}(X) \; , \hsp{3}
 \s{\s{g}}_{(3)}(X)  \;\equiv_{\pi^{(s - t_1) - t_0}} \; g_{(3)}(X) \; , \hsp{3} \\
 \s{h}_{(2)}(X) \; \equiv_{\pi^{2((s - t_1) - t_0)}}\;  \s{\s{g}}_{(2)}(X) \cdot \s{\s{g}}_{(3)}(X) \; .
\end{array}
\]
Altogether, the two subsequent applications of Lemma~\ref{KochHenselStepPrae}.(1) for two factors yield\hypertarget{ii}{}
\[
\begin{array}{l}
 \s{h}_{(1)}(X)      \;\equiv_{\pi^{s - t_1}}\;  g_{(1)}(X) \; , \;\;
 \s{\s{g}}_{(2)}(X)  \;\equiv_{\pi^{s - t_1 - t_0}} \;  g_{(2)}(X) \; , \;\;
 \s{\s{g}}_{(3)}(X)  \;\equiv_{\pi^{s - t_1 - t_0}} \;  g_{(3)}(X) \\
f(X) \;  \equiv_{\pi^{2(s - t_1)}} \;\s{h}_{(1)}(X) \cdot \s{h}_{(2)}(X)        
     \; \equiv_{\pi^{2(s - t_1 - t_0)}} \; \s{h}_{(1)}(X) \cdot \s{\s{g}}_{(2)}(X) \cdot \s{\s{g}}_{(3)}(X)   \;.  \\
\end{array}
\leqno \text{(ii)} 
\]
Comparing the result (\hyperlink{i}{i}) of Lemma~\ref{KochHenselStepPrae}.(1) for three factors with the result (\hyperlink{ii}{ii}) 
of two subsequent applications of Lemma~\ref{KochHenselStepPrae}.(1) for two factors, both methods essentially yield a precision of $s - t$ for the factors and a 
precision of $2(s - t)$ for the product decomposition.
\end{Example}

\subsection{Case $f(X)\equiv_\pi X^M$}
\label{SecSp}

Let $f(X)\in R[X]$ be a monic polynomial. Write $M := \deg f$. Suppose that $f(X)\equiv_\pi X^M$.

Let \mbox{$n\geq 1$}. Suppose given monic polynomials \mbox{$g_{(1)}(X) , \dots , g_{(n)}(X) \in R[X]$} with degree $\ge 1$.
Write \linebreak
\mbox{$m_{(k)} \,:=\, \deg(g_{(k)})$ and $M_{(k)} := M - m_{(k)}$ for $k\in [1,n]$}. 
Suppose the ordering to be chosen such that
\mbox{$m_{(1)} \le m_{(2)}\le \dots \le m_{(n)}$}
and that \mbox{$\Res(g_{(1)},\dots, g_{(n)})\ne 0$}.
Let
\mbox{$
t  := \val_\pi\!\big(\Res(g_{(1)},\dots, g_{(n)})\big)
$},
$
t' := e' \; :=\; \val_\pi\!\big(\Res(g_{(1)}\,,\,\dots\,,\, g_{(n)})\big) - \sum_{j\in [1,n-1]} \big((n-j)m_{(j)} - 1\big) \; .
$
Let $s\geq t + t' + 1$. Suppose that
$
f(X) \equiv_{\pi^{s}} \prod\limits_{k \in [1,n]} g_{(k)}(X)\; .
$
We remark that $g_{(k)}(X) \equiv_\pi X^{m_{(k)}}$ for $k\in [1,n]$.

(Note that we may replace $s\geq t + t' + 1$ by $s \ge \val_\pi(\Delta(f)) + 1$ if $\Delta(f)\ne 0$; cf.\ Proposition~\ref{bewertung}.)

\begin{Lemma}
\label{LemSp2}
Let $\ell\geq 1$. Let $h_{(1)}(X) \,,\, \dots \,,\, h_{(\ell)}(X) \in R[X]$ be monic polynomials of degree $\ge 1$. 
Write $\chi_{(k)} := \deg(h_{(k)})$ for $k\in [1,\ell]$. Write $\chi := \sum_{k\in [1,\ell]} \chi_{(k)}\,$. Suppose the ordering to be chosen such
that $\chi_{(1)}\le\chi_{(2)}\le\dots\le\chi_{(\ell)}\,$. Suppose that $h_{(k)}(X) \equiv_\pi X^{\chi_{(k)}}$ for $k\in [1,\ell]$.
Write $\prod_{k\in [1,\ell]} h_{(k)}(X) =: \sum_{i\in [0,\chi]} b_i X^i$ with $b_i\in R$ for $i\in [0,\chi]$.

Then 
$
\val_\pi(b_i) \ge \ell - \max\{\,j\in [0,\ell]\, :\, \chi_{(1)} + \cdots + \chi_{(j)} \le i\,\}
$
for $i\in [0,\, \chi]$.
\end{Lemma}

\begin{proof}
Write $h_{(k)}(X) =: \sum_{i\in [0,\chi_{(k)}]} h_{(k)i} X^i$ for $k\in [1,\ell]$, where $h_{(k)i}\in R$ for $i\in [0,\chi_{(k)}]$. We have
$
\; b_i \; =\; 
\sum_{i_{(k)}\in [0,\chi_{(k)}]\;\text{for $k\in [1,\ell]$,}\; i_{(1)} + \cdots + i_{(\ell)} \,=\, i}\;\;
\prod_{k\in [1,\ell]} h_{(k)i_{(k)}}\; .
$
So it suffices to show that 
$
\val_\pi\big(\prod_{k\in [1,\ell]} h_{(k)i_{(k)}}\big)\;\overset{!}{\ge}\; \ell - \max\{\,j\in [0,\ell]\, :\, \chi_{(1)} + \cdots + \chi_{(j)} \le i\,\}
$
for all occurring summands. Since $\val_\pi(h_{(k)i_{(k)}}) \ge 1$ if $i_{(k)} \in [0,\chi_{(k)}-1]$, it remains to show that for such a summand,
we have
$
|\{\,k\in [1,\ell]\, : \,i_{(k)} = \chi_{(k)} \,\}| \;\overset{!}{\le}\; \max\{\,j\in [0,\ell] \, :\, \chi_{(1)} + \cdots + \chi_{(j)} \le i\,\}\; .
$

{\it Assume} that 
$
|\{\,k\in [1,\ell]\, : \,i_{(k)} = \chi_{(k)} \,\}| \; > \; \max\{\,j\in [0,\ell]\, :\, \chi_{(1)} + \cdots + \chi_{(j)} \le i\,\}\; .
$
Write $H := \{\,k\in [1,\ell]\, : \,i_{(k)} = \chi_{(k)} \,\}\subseteq [1,\ell]$. 
Then $\ell\ge |H| > \max\{\,j\in [0,\ell]\, :\, \chi_{(1)} + \cdots + \chi_{(j)} \le i\,\}$, whence $\chi_{(1)} + \cdots + \chi_{(|H|)} > i$. So, 
using $\chi_{(1)}\le\chi_{(2)}\le\dots\le\chi_{(\ell)}$, we get
$
i 
 =    i_{(1)} + \cdots + i_{(\ell)}                                          
 =    (\sum_{k\in H} i_{(k)}) + (\sum_{k\in [1,\ell]\minus H} i_{(k)})       
 \ge  \sum_{k\in H} i_{(k)}                                                  
 =    \sum_{k\in H} \chi_{(k)}                                               
 \ge  \sum_{k\in [1,|H|]} \chi_{(k)} 
 >    i                                                         
$. {\it Contradiction.}
\end{proof}

\begin{Lemma}
\label{LemSp3}\indent\vsp{-1} 
\begin{itemize}
\item[\rm (1)] We have $e'\ge 0$.
\item[\rm (2)] Given $y\in \pi^{e'} R^{1\ti M}$, there exists $x\in R^{1\ti M}$ such that $x\Sylv(g_{(1)}\,,\,\dots\,,\, g_{(n)}) = y$.
\item[\rm (3)] Suppose given $u\ge e'$ and $x\in R^{1\ti M}$ such that $x\Sylv(g_{(1)}\,,\,\dots\,,\, g_{(n)})\in R^{1\ti M} \pi^u$. 

Then $x\in R^{1\ti M} \pi^{u - e'}$.
\end{itemize}
\end{Lemma}

\begin{proof}
Write $\prod\limits_{j\in[1,n]\minus\{k\}} g_{(j)}(X) \;=:\; \sum\limits_{i\in[0,M_{(k)}]} a_{(k)i}X^{i}$ for $k\in [1,n]$.

Suppose given $i\in [1,M]$. Write
$
d_i \; := \; (n - 1) - \max\{\,j\in [0,n-1]\, :\, m_{(1)} + \cdots + m_{(j)} \le i - 1\,\}\; .
$
Note that $d_\xi \ge d_\eta$ for $1\le\xi\le\eta\le M$. By Lemma~\ref{LemSp2}, we have
$
\val_\pi(a_{(k)i-1})\;\ge\; d_i
$
for $k\in [1,n]$, since the sequence of degrees of the polynomials $g_{(j)}(X)$, with $g_{(k)}(X)$ omitted, is entrywise bounded below by 
the sequence of degrees of the polynomials $g_{(j)}(X)$, i.e.\ by the sequence of the~$m_{(j)}\,$. 
It follows that
$
\val_\pi(a_{(k)\xi-1})\;\ge\; d_\xi\;\ge\; d_i
$
for $k\in [1,n]$ and $\xi\in [1,i]$. Hence $\pi^{d_i}$ divides column number 
$i$ of $\Sylv(g_{(1)}\,,\,\dots\,,\, g_{(n)})\,$; cf.\ Definition~\ref{Def1}.

We have
\[
\begin{array}{rl}
  & \hsp{-10} d_2 + \cdots + d_M \;=\; \sum_{i\in [2,M]} \left((n - 1) - \max\{\,j\in [0,n-1]\, :\, m_{(1)} + \cdots + m_{(j)} \le i - 1\,\}\right) \\
= & (M - 1)(n - 1) - \sum_{i\in [1,M-1]}\max\{\,j\in [0,n-1]\, :\, m_{(1)} + \cdots + m_{(j)} \le i\,\}                                  \\
= & (M - 1)(n - 1) - \sum_{j\in [0,n-1]}\, j\cdot |\,[\,m_{(1)} + \cdots + m_{(j)}\,,\, m_{(1)} + \cdots + m_{(j)} + m_{(j+1)} - 1\,]\,| \\
= & (M - 1)(n - 1) - \sum_{j\in [0,n-1]}\, j m_{(j+1)} \; =\; (M - 1)(n - 1) - \sum_{j\in [1,n]}\, (j-1) m_{(j)}                         \\
= & (M - 1)(n - 1) + M - \sum_{j\in [1,n]}\, j m_{(j)} \;=\; 1 + nM - n - \sum_{j\in [1,n]}\, j m_{(j)}                                  \\
= & 1 + \sum_{j\in [1,n]} \big((n-j)m_{(j)} - 1\big) \;=\; \sum_{j\in [1,n-1]} \big((n-j)m_{(j)} - 1\big)\; ,                            \\
\end{array}
\]
whence
\mbox{$
\val_\pi(\det \Sylv(g_{(1)}\,,\,\dots\,,\, g_{(n)})) - (d_2 + \cdots + d_M) \; =\; e'.
$}
So assertion~(2) follows by Lemma~\ref{LinearAlgebra1}.(3), assertion~(3) follows by Lemma~\ref{LinearAlgebra2}.(3); moreover, assertion~(1) follows by Remark~\ref{RemMinors}.
\end{proof}

\begin{Lemma}
\label{KochHenselStepPraeSp}
\indent\vsp{-1}
\begin{itemize}
\item[\rm (1)] There exist monic polynomials $\s{g}_{(1)}(X) , \dots , \s{g}_{(n)}(X) \in R[X]$
such that\\
$
\s{g}_{(k)}(X) \equiv_{\pi^{s-t'}} g_{(k)}(X) 
$
for $k \in [1,n]$ and
$
f(X) \,\equiv_{\pi^{2(s - t')}}\, \prod\limits_{k \in [1,n]} \s{g}_{(k)}(X)\; .
$

We call such a tuple $(\s{g}_{(k)}(X))_k$ of polynomials an {\em admissible$'$ lift} of $(g_{(k)}(X))_k$ with respect to $s$.
We have
$
\val_\pi(\Res(\s{g}_{(1)}\,,\,\dots\,,\, \s{g}_{(n)})) \,=\, t
$
for any admissible$'$ lift $(\s{g}_{(k)}(X))_k$ of $(g_{(k)}(X))_k$ with respect to $s$. 

\item[\rm (2)] Suppose given $r\in [0,s - 2t']$. 
Suppose given monic polynomials $\s{g}_{(1)}(X) \,,\, \dots \,,\, \s{g}_{(n)}(X),$ $\s{h}_{(1)}(X), \dots , \s{h}_{(n)}(X) \in R[X]$
such that
$
\s{g}_{(k)}(X) \equiv_{\pi^{s-t'}} g_{(k)}(X) 
$
and
$
\s{h}_{(k)}(X) \equiv_{\pi^{s-t'}} g_{(k)}(X)
$
for $k \in [1,n]$ and
\smash{$
\prod\limits_{k \in [1,n]} \s{g}_{(k)}(X)\equiv_{\pi^{2(s - t') - r}} \prod\limits_{k \in [1,n]} \s{h}_{(k)}(X).
$}\ru{4.5}
Then 
$
\s{g}_{(k)}(X) \equiv_{\pi^{2s - 3t' - r}} \s{h}_{(k)}(X)
$
for $k\in [1,n]$.

In particular, considering the case $r = 0$, two admissible$'$ lifts with respect to $s$ as in {\rm (1)} are mutually congruent modulo 
$\pi^{2s - 3t'}R[X]$.
\end{itemize}
\end{Lemma}

In the following proof, we shall use the notation of \S~\ref{SecRes}.

\begin{proof}
{\it Ad~{\rm (1).}} {\it Existence of admissible$'$ lift.} We make the ansatz
$
\s{g}_{(k)}(X) \,=\, g_{(k)}(X)+\pi^{s-t'}u_{(k)}(X) 
$
for $k \in [1,n]$ with $u_{(k)}(X) \,\in\, R[X]$ and $\deg u_{(k)} < \deg g_{(k)} = m_{(k)}$ for $k \in [1,n]$. Thus we require that
\[
\begin{array}{rll}
f(X) 
& \overset{!}{\equiv}_{\pi^{2(s - t')}} & \prod\limits_{k \in [1,n]} \s{g}_{(k)}(X)                                                                    
\; =\;                                      \prod\limits_{k \in [1,n]} (g_{(k)}(X)+\pi^{s-t'}u_{(k)}(X))                                               \\
& \equiv_{\pi^{2(s - t')}}              & \prod\limits_{k \in [1,n]} g_{(k)}(X) 
                                          + \pi^{s-t'}\sum\limits_{k\in [1,n]}u_{(k)}(X)\cdot\prod\limits_{\ell\in [1,n]\minus\{k\}} g_{(\ell)}(X) \; .\\
\end{array}
\]

Let $b(X):= \pi^{t'-s}(f(X)-\prod\limits_{k \in [1,n]}g_{(k)}(X))$. Since $f(X) \equiv_{\pi^{s}} \prod\limits_{k \in [1,n]}g_{(k)}(X)$, we get 
$b(X) \equiv_{\pi^{t'}} 0$. So our requirement reads
$
b(X) \,\overset{!}\equiv_{\pi^{s - t'}} \sum\limits_{k \in [1,n]} u_{(k)}(X) \,\cdot\, 
\prod\limits_{\ell \in [1,n]\minus\{k\}} g_{(\ell)}(X)\, .
$
So it suffices to find $u_{(k)}(X) \in R[X]$ for $k \in [1,n]$ as above such that
$
 b(X) \,\overset{!}{=} \sum\limits_{k \in [1,n]} u_{(k)}(X) \,\cdot\, \prod\limits_{\ell \in [1,n]\minus\{k\}} g_{(\ell)}(X) \, .
$

Writing 
$b(X)                                                  =:   \sum\limits_{i\geq 0} \beta_i X^i$, 
$\prod\limits_{\ell\in[1,n]\minus\{k\}} g_{(\ell)}(X)  =:   \sum\limits_{i \geq 0} a_{(k)i}X^{i}$,
$u_{(k)}(X)                                            =:   \sum\limits_{i \geq 0} u_{(k)i}X^{i}$
for $k\in [1,n]$, where $\beta_{i}\,,\, a_{(k)i}\,,\, u_{(k)i}  \,\in\, R$ for $i\ge 0$, a comparison of coefficients shows that it suffices to 
find 
\[
U \,:=\, (\ubrt{\ru{-4}u_{(1)\,0} \, \dots \, u_{(1)\,m_{(1)}-1}}\; \ubrt{\ru{-4}u_{(2)\,0} \, \dots \, u_{(2)\,m_{(2)}-1}} \;\;\,\dots\,\;\; 
          \ubrt{\ru{-4}u_{(n)\,0} \,\dots\, u_{(n)\,m_{(n)}-1}})\,\in\, R^{1\ti M}
\]
such that
$
U \cdot \Sylv(g_{(1)} \,,\, \dots \,,\,g_{(n)}) \;\aufgl{!} \; (\beta_0 \,\dots\, \beta_{M-1}).
$
Note that $(\beta_0\,\dots\,\beta_{M-1}) \in \pi^{t'} R^{1\ti M}$ since $b(X) \equiv_{\pi^{t'}} 0$. So $U$ exists as
required by Lemma~\ref{LemSp3}.(2).

{\it Valuation of resultant.} Since $\s{g}_{(k)}(X) \equiv_{\pi^{s - t'}} g_{(k)}(X)$ for $k\in [1,n]$, Remark~\ref{bewertungRes} implies that
$
\Res(\s{g}_{(1)},\dots, \s{g}_{(n)})\,\equiv_{\pi^{s - t'}}\,\Res(g_{(1)},\dots, g_{(n)})\, .
$
Since $s - t' \ge t + 1 = \val_\pi(\Res(g_{(1)}\,,\,\dots\,,\, g_{(n)})) + 1$, this implies
$
\val_\pi(\Res(\s{g}_{(1)}\,,\,\dots\,,\, \s{g}_{(n)}))\,=\,\val_\pi(\Res(g_{(1)}\,,\,\dots\,,\, g_{(n)}))\;=\; t\, .
$

{\it Ad~{\rm (2).}} Writing
$\s{g}_{(k)}(X) =: g_{(k)}(X) + \pi^{s - t'} u_{(k)}(X)$
and
$\s{h}_{(k)}(X) =: g_{(k)}(X) + \pi^{s - t'} v_{(k)}(X)$
for $k\in [1,n]$, where $u_{(k)}(X),\,v_{(k)}(X)\,\in\, R[X]$, we obtain $\deg u_{(k)}(X) < \deg g_{(k)}(X) = m_{(k)}\,$, 
since $\s{g}_{(k)}(X)$ and $g_{(k)}(X)$ are monic polynomials of the same degree; likewise, we obtain \mbox{$\deg v_{(k)}(X) < m_{(k)}\,$}.

We have to show that $u_{(k)}(X)\overset{!}{\equiv}_{\pi^{s - 2t' - r}} v_{(k)}(X)$ for $k \in [1,n]$. We have
\[
\begin{array}{c}
\prod\limits_{k \in [1,n]} g_{(k)}(X) + \pi^{s-t'}\sum\limits_{k\in [1,n]}u_{(k)}(X)\cdot\prod\limits_{\ell\in [1,n]\minus\{k\}} g_{(\ell)}(X)    
\;\equiv_{\pi^{2(s - t')}}\; \prod\limits_{k \in [1,n]} (g_{(k)}(X)+\pi^{s-t'}u_{(k)}(X))                                                   \\                
\; = \; \prod\limits_{k \in [1,n]} \s{g}_{(k)}(X) \;\equiv_{\pi^{2(s - t') - r}}\;\prod\limits_{k \in [1,n]} \s{h}_{(k)}(X)                   \\
\; = \; \prod\limits_{k \in [1,n]} (g_{(k)}(X)+\pi^{s-t'}v_{(k)}(X)) 
\;\equiv_{\pi^{2(s - t')}}\; \prod\limits_{k \in [1,n]} g_{(k)}(X) 
                                 + \pi^{s-t'}\sum\limits_{k\in [1,n]}v_{(k)}(X)\cdot\prod\limits_{\ell\in [1,n]\minus\{k\}} g_{(\ell)}(X) \; .\\
\end{array}
\]
The difference yields
$
\sum\limits_{k\in [1,n]} (u_{(k)}(X) - v_{(k)}(X))\cdot\prod\limits_{\ell\in [1,n]\minus\{k\}} g_{(\ell)}(X) \,\equiv_{\pi^{s - t'-r}}\, 0\, .
$

Writing
$
w_{(k)}(X) \;:=\; u_{(k)}(X) - v_{(k)}(X)
$
for $k\in [1,n]$, this reads\hypertarget{ast.0}%
\[
\sum\limits_{k\in [1,n]} w_{(k)}(X)\cdot\prod\limits_{\ell\in [1,n]\minus\{k\}} g_{(\ell)}(X) \;\equiv_{\pi^{s - t' - r}}\; 0 \; .
\leqno (\ast)
\]
Writing 
$
w_{(k)}(X) \,=:\, \sum\limits_{i\geq 0} w_{(k)i}X^{i}
$
for $k\in [1,n]$, and
\[
W \,:=\, (\ubrt{\ru{-4}w_{(1)\,0} \, \dots \, w_{(1)\,m_{(1)}-1}}\; \ubrt{\ru{-4}w_{(2)\,0} \, \dots \, w_{(2)\,m_{(2)}-1}} \;\;\,\dots\,\;\; 
          \ubrt{\ru{-4}w_{(n)\,0} \,\dots\, w_{(n)\,m_{(n)}-1}})\,\in\, R^{1\ti M}\; ,
\]
we have to show that $W\overset{!}{\in}\pi^{s - 2t' - r} R^{1\ti M}$. From \hyperlink{ast.0}{$(\ast)$}, we obtain
$
W \cdot \Sylv(g_{(1)} \,,\, \dots \,,\,g_{(n)}) \,\in\, \pi^{s - t' - r}R^{1\ti M}\, .
$
Note that $s - t' - r\ge t'= e'$. 
So we can infer by Lemma~\ref{LemSp3}.(3) that $W\in \pi^{(s - t' - r) - t'}R^{1\ti M} = \pi^{s - 2t' - r}R^{1\ti M}$.
\end{proof}

\begin{Theorem}
\label{KochHenselSp}
Suppose $R$ to be a complete discrete valuation ring. Recall that $f(X)\in R[X]$ is a monic polynomial, that $M = \deg f$, that $f(X)\equiv_\pi X^M$, that 
$g_{(1)}(X) , \dots , g_{(n)}(X) \in R[X]$ are monic polynomials, that $t = \val_\pi\!\big(\Res(g_{(1)},\dots, g_{(n)})\big)$, that $m_{(k)} = \deg(g_{(k)})$ for $k\in [1,n]$, 
ordered such that $m_{(1)} \le \dots \le m_{(n)}\,$,
that $t' = \val_\pi\!\big(\Res(g_{(1)}\,,\,\dots\,,\, g_{(n)})\big) - \sum_{j\in [1,n-1]} \big((n-j)m_{(j)} - 1\big)$, that $s\geq t + t' + 1$ and that 
$f(X) \equiv_{\pi^{s}} \prod\limits_{k \in [1,n]} g_{(k)}(X)$.

Then there exist unique monic polynomials $\ih{g}_{(1)}(X) \,,\, \dots \,,\, \ih{g}_{(n)}(X) \,\in\, R[X]$ such that 
$
\ih{g}_{(k)}(X) \,\equiv_{\pi^{s-t'}}\, g_{(k)}(X)
$
for $k \in [1,n]$ and
$
f(X) \;=\; \prod\limits_{k \in [1,n]} \ih{g}_{(k)}(X) \; .
$
\end{Theorem}

\begin{proof}
{\it Existence.} Since $R$ is complete, it suffices to show that there exist monic polynomials
$\s{g}_{(1)}(X) , \dots , \s{g}_{(n)}(X) \in R[X]$ such that 
$
f(X) \equiv_{\pi^{s+1}} \prod\limits_{k \in [1,n]} \s{g}_{(k)}(X)
$
and
$
\s{g}_{(k)}(X) \equiv_{\pi^{s-t'}} g_{(k)}(X)
$
for $k \in [1,n]$, and
$
\val_{\pi}(\Res(\s{g}_{(1)},\dots, \s{g}_{(n)})) = t. 
$
This follows from Lemma~\ref{KochHenselStepPraeSp}.(1) since $2(s - t') \ge s+1$.

{\it Uniqueness.} Given $\ih{g}_{(1)}(X) , \dots , \ih{g}_{(n)}(X)\,,\,\ih{h}_{(1)}(X), \dots , \ih{h}_{(n)}(X) \,\in\, R[X]$, all monic, such that
$
\ih{g}_{(k)}(X) \equiv_{\pi^{s-t'}} g_{(k)}(X)\equiv_{\pi^{s-t'}}\ih{h}_{(k)}(X)
$
for $k \in [1,n]$ and
\smash{$
\prod\limits_{k \in [1,n]} \ih{g}_{(k)}(X) = f(X) = \prod\limits_{k \in [1,n]} \ih{h}_{(k)}(X),
$}
we have to show that $\ih{g}_{(k)}(X)\overset{!}{=} \ih{h}_{(k)}(X)$ for $k \in [1,n]$.

Note that $\val_{\pi}(\Res(\ih{g}_{(1)}\,,\,\dots\,,\, \ih{g}_{(n)})) = t = \val_{\pi}(\Res(\ih{h}_{(1)}\,,\,\dots\,,\, \ih{h}_{(n)}))$
by Lemma~\ref{KochHenselStepPraeSp}.(1).

Let $s_1 := s$. Both $(\ih{h}_{(k)}(X))_k$ and $(\ih{g}_{(k)}(X))_k$ are admissible$'$ lifts of $(\ih{g}_{(k)}(X))_k$ with respect to $s_1$ in
the sense of Lemma~\ref{KochHenselStepPraeSp}.(1), since $\ih{h}_{(k)}(X) \equiv_{\pi^{s_1 - t'}} \ih{g}_{(k)}(X)$ for $k\in [1,n]$ and the 
other required congruences are verified using equalities.
So Lemma~\ref{KochHenselStepPraeSp}.(2) yields
$
\ih{h}_{(k)}(X) \,\equiv_{\pi^{2(s_1 - t') - t'}}\, \ih{g}_{(k)}(X) 
$
for $k\in [1,n]$.

Let $s_2 := 2(s_1 - t')$. Note that $s_2 = s_1 + (s_1 - 2t') > s_1\,$. Both $(\ih{h}_{(k)}(X))_k$ and $(\ih{g}_{(k)}(X))_k$ are 
admissible$'$ lifts of $(\ih{g}_{(k)}(X))_k$ with respect to $s_2$, since $\ih{h}_{(k)}(X)\equiv_{\pi^{s_2 - t'}}\ih{g}_{(k)}(X)$ for $k\in [1,n]$.
So Lemma~\ref{KochHenselStepPraeSp}.(2) yields
$
\ih{h}_{(k)}(X) \,\equiv_{\pi^{2(s_2 - t') - t'}}\, \ih{g}_{(k)}(X)
$
for $k\in [1,n]$.

Let $s_3 := 2(s_2 - t')$. Note that $s_3 = s_2 + (s_2 - 2t') > s_2 + (s_1 - 2t') > s_2\,$. Continue as above.

This yields a strictly increasing sequence $(s_\ell)_{\ell\,\ge\, 1}$ of integers such that 
$
\ih{h}_{(k)}(X) \,\equiv_{\pi^{s_\ell - t'}}\, \ih{g}_{(k)}(X) 
$
for $k\in [1,n]$ and $\ell\ge 1$. Hence
$
\ih{h}_{(k)}(X) \,=\, \ih{g}_{(k)}(X) 
$
for $k\in [1,n]$.
\end{proof}

\begin{Example}
\label{ExComparisonSp}\rm
Suppose that $n = 3$ and $s\ge 2t+1$. Write
$
t_0 := \val_\pi(\Res(g_{(2)} \,,\, g_{(3)}))
$
and
$
t_1 :=  \val_\pi(\Res(g_{(1)} \,,\, g_{(2)}g_{(3)})) 
$.
Lemma~\ref{res=prod} gives $\Res(g_{(1)} \,,\, g_{(2)} \,,\, g_{(3)}) = \Res(g_{(1)} \,,\, g_{(2)}g_{(3)}) \cdot \Res(g_{(2)} \,,\, g_{(3)})$,
whence $t = t_1 + t_0\,$. In particular, $t_0$ and $t_1$ are finite.

Denote $t' := t - 2m_{(1)} - m_{(2)} + 2$, $t'_0 := t_0 - m_{(2)} + 1$ and $t'_1 := t_1 - m_{(1)} + 1$. So $s\ge 2t+1 \ge t + t' + 1$
and $t' = t'_1 + t'_0 - m_{(1)}\,$.

We can apply Lemma~\ref{KochHenselStepPraeSp}.(1) to  $f(X) \equiv_{\pi^{s}} g_{(1)}(X)\cdot g_{(2)}(X)\cdot g_{(3)}(X)$ 
to obtain monic polynomials $\s{g}_{(1)}(X) \,,\, \s{g}_{(2)}(X)\,,\, \s{g}_{(3)}(X)\,\in\, R[X]$ such 
that\hypertarget{is}{}
\[
\s{g}_{(k)}(X) \;\equiv_{\pi^{s-t'}} \;g_{(k)}(X) \hsp{2} \text{for $k \in [1,3]$} \; ,\hsp{3}         
 f(X)          \;\equiv_{\pi^{2(s - t')}} \; \s{g}_{(1)}(X)\cdot\s{g}_{(2)}(X)\cdot\s{g}_{(3)}(X) \, .  
\leqno \text{(i$'$)}
\]
We can also apply Lemma~\ref{KochHenselStepPraeSp}.(1) first to the factorisation $f(X) \equiv_{\pi^{s}} g_{(1)}(X)\cdot (g_{(2)}(X)\cdot g_{(3)}(X))$ and
then to the resulting factorisation of the second factor into $g_{(2)}(X)\cdot g_{(3)}(X)$ modulo a certain power of $\pi$. 

We have $s \geq 2t+1 \geq 2t_1+1 \geq t_1 + t'_1 + 1$. So Lemma~\ref{KochHenselStepPraeSp}.(1) gives monic polynomials 
$\s{h}_{(1)}(X) \,,\, \s{h}_{(2)}(X) \,\in\,R[X]$ such that
\[
 \s{h}_{(1)}(X) \; \equiv_{\pi^{s-t'_1}} \; g_{(1)}(X) \; , \;\;
 \s{h}_{(2)}(X) \; \equiv_{\pi^{s-t'_1}} \; g_{(2)}(X)\cdot g_{(3)}(X)  \;,\;\;    
 f(X)           \; \equiv_{\pi^{2(s - t'_1)}}  \; \s{h}_{(1)}(X) \cdot \s{h}_{(2)}(X) \;. 
\]
We have $s\geq 2t+1 \geq 2t - m_{(2)} - m_{(1)} + 3 - t_1 = (2t_0 - m_{(2)} + 1) + (t_1 - m_{(1)} + 1) + 1 
= (t_0 + t'_0) + t'_1 + 1$ and thus $s - t'_1 \ge t_0 + t'_0 + 1$. So Lemma~\ref{KochHenselStepPraeSp}.(1) gives
monic polynomials $\s{\s{g}}_{(2)}(X)\,,\,\s{\s{g}}_{(3)}(X)\,\in\,R[X]$ such that 
\[
 \begin{array}{c}
 \s{\s{g}}_{(2)}(X)  \;\equiv_{\pi^{(s - t'_1) - t'_0}} \; g_{(2)}(X) \; , \hsp{3}
 \s{\s{g}}_{(3)}(X)  \;\equiv_{\pi^{(s - t'_1) - t'_0}} \; g_{(3)}(X) \; , \hsp{3} \\
 \s{h}_{(2)}(X) \; \equiv_{\pi^{2((s - t'_1) - t'_0)}}\;  \s{\s{g}}_{(2)}(X) \cdot \s{\s{g}}_{(3)}(X) \; .
\end{array}
\]
Altogether, the two subsequent applications of Lemma~\ref{KochHenselStepPraeSp}.(1) for two factors yield\hypertarget{iis}{}
\[
\begin{array}{l}
 \s{h}_{(1)}(X)      \;\equiv_{\pi^{s - t'_1}}\;  g_{(1)}(X) \; , \;\;
 \s{\s{g}}_{(2)}(X)  \;\equiv_{\pi^{s - t'_1 - t'_0}} \;  g_{(2)}(X) \; , \;\;
 \s{\s{g}}_{(3)}(X)  \;\equiv_{\pi^{s - t'_1 - t'_0}} \;  g_{(3)}(X) \\
f(X) \;  \equiv_{\pi^{2(s - t'_1)}} \;\s{h}_{(1)}(X) \cdot \s{h}_{(2)}(X)        
     \; \equiv_{\pi^{2(s - t'_1 - t'_0)}} \; \s{h}_{(1)}(X) \cdot \s{\s{g}}_{(2)}(X) \cdot \s{\s{g}}_{(3)}(X)   \;.  \\
\end{array}
\leqno\text{(ii$'$)}
\]
Comparing the result (\hyperlink{is}{i$'$}) of Lemma~\ref{KochHenselStepPraeSp}.(1) for three factors with the result (\hyperlink{iis}{ii$'$}) 
of two subsequent applications of Lemma~\ref{KochHenselStepPraeSp}.(1) for two factors, the former method yields a precision of $s - t'$ for the factors and a precision of 
$2(s - t')$ for the product decomposition, the latter method yields a precision of $s - t'_0 - t'_1$ for the factors and a precision of 
$2(s - t'_0 - t'_1)$ for the product decomposition. Since $t' = t'_1 + t'_0 - m_{(1)} < t'_1 + t'_0$, the former method yields a higher precision.
\end{Example}

\section{Examples}
\label{SecEx}

To illustrate Theorem~\ref{KochHensel} we consider some polynomials in the complete discrete valuation ring $\Z_p$ for a prime number $p$.
Given a polynomial in $\Z[X] \subseteq \Z_p[X]$ and a 
factor decomposition in $\Z[X]$ to a certain $p$-adic precision, the method of the proof of Lemma~\ref{KochHenselStepPrae} gives a factor decomposition in $\Z[X]$ to a higher 
$p$-adic precision. We use the notation of Lemma~\ref{KochHenselStepPrae}. 

Write $g_{(k)}(X)=: \!\!\sum\limits_{j\in[0,m_{(k)}]}\!\! c_{(k)j} X^j$ and 
$\s{g}_{(k)}(X) =: \!\!\sum\limits_{j\in[0,m_{(k)}]}\!\! \s{c}_{(k)j} X^j$
for $k \in [1,n]$, where $c_{(k)j}\,,\s{c}_{(k)j}\in\Z$. Write $f(X) - \prod_{k\in[1,n]} g_{(k)}(X) =: \sum\limits_{j\in[0,M]} w_j X^j$, where $w_j\in\Z$.
Let $s$ be the current precision, i.e.
\[
s \; :=\; \min\; \{\, \val_\pi(w_j) \,:\, j\in [0,M]\,\}\; .
\]
In the respective initial step of the examples below, we ensure that $s\ge\val_p(\Delta(f)) + 1$. Write 
\[
 s' \;:=\; \min\; \{\, \val_\pi ( c_{(k)j} - \s{c}_{(k)j} ) \,:\, k \in [1,n],\, j \in [0,m_{(k)}] \,\}\; .           
\]
By Lemma~\ref{KochHenselStepPrae}, we have $s' \ge s - t$. Let the {\it defect} be $s - s'$. The defect is bounded above by $t$. 

If $f(X) \equiv_\pi X^M$ and the degrees of the factors $g_{(k)}(X)$ are sorted increasingly, then the defect $s - s'$ is bounded above
by $t'\,$; cf.\ Lemma~\ref{KochHenselStepPraeSp}.

The following examples have been calculated using {\sc Magma} \bfcit{Magma}.

\begin{Example}
 \label{Ex_u51}\rm
We consider the polynomial 
\[
f(X) = X^3 + X^2 - 2X + 8
\]
at $p = 2\,$; it has been used as an example by {\sc Dedekind} and {\sc Koch}; cf.\ \bfcite{De30}{p.\ 225}, \bfcite{Ko97}{\S 3.12, introduction to \S 4, \S 4.4}.

\begin{tabular}{p{8cm}p{8cm}}

We start with initial precision $s = 3$.
We consider the development of the factors $g_{(1)}(X)$, $g_{(2)}(X)$, $g_{(3)}(X)$ during steps 1 to 6, starting with the initial factorisation during step~1.

\begin{center}
\begin{footnotesize}
\begin{tabular}{|c|l|} \hline

step 1 & $g_{(1)}(X) = X$                                                 \\
        & $g_{(2)}(X) = X + 2$                                             \\
        & $g_{(3)}(X) = X + 7$                                             \\
                                                                   \hline

step 2  & $g_{(1)}(X) = X + 12$                                            \\
        & $g_{(2)}(X) = X + 14$                                            \\
        & $g_{(3)}(X) = X + 7$                                             \\
                                                                   \hline
 step 3 & $g_{(1)}(X) = X + 52$                                            \\
        & $g_{(2)}(X) = X + 54$                                            \\
        & $g_{(3)}(X) = X + 23$                                            \\
                                                                   \hline
 step 4 & $g_{(1)}(X) = X + 980$                                           \\
        & $g_{(2)}(X) = X + 470$                                           \\
        & $g_{(3)}(X) = X + 599$                                           \\
                                                                   \hline
 step 5 & $g_{(1)}(X) = X + 167380$                                        \\
        & $g_{(2)}(X) = X + 224214$                                        \\
        & $g_{(3)}(X) = X + 132695$                                        \\
                                                                   \hline
step 6  & $g_{(1)}(X) = X + 1339592148$                                    \\
        & $g_{(2)}(X) = X - 4836725802$                                    \\
        & $g_{(3)}(X) = X + 3497133655$                                    \\
                                                                   \hline

\end {tabular} 
\end{footnotesize}
\end{center} &

We obtain the following results in the first 10 steps. The defect is bounded above by $t = 1$.

\begin{center}
\begin{footnotesize}
\begin{tabular}{|l|c|c|c|c|} \hline
step  & current      & defect \vsp{-1}        \\ 
      & precision    &        \vsp{-1}        \\
      & $s$          & $s-s'$                 \\
                                                            \hline
 1    & 3            & 1                      \\
                                                            \hline
 2    & 4            & 1                      \\
                                                            \hline
 3    & 6            & 1                      \\
                                                           \hline
 4    & 10           & 1                      \\
                                                           \hline
 5    & 18           & 1                      \\
                                                           \hline
 6    & 34           & 1                      \\
                                                            \hline
 7    & 66           & 1                      \\
                                                            \hline
 8    & 130          & 1                      \\
                                                           \hline
 9    & 258          & 1                      \\
                                                           \hline
10    & 514          & 1                      \\
                                                           \hline
\end{tabular}
\end{footnotesize}
\end{center}

The defect seems to be constant with value~$1$. 
We observe that the defect is maximal. 

Note that in step 1, the precision grows only by~$1$. \\
\end{tabular}
\end{Example}

\begin{Example}
\label{Ex_u37}\rm\indent\vsp{-0.2}

\begin{tabular}{p{10.5cm}p{6cm}}
We consider the polynomial 
\[
f(X) \;=\; X^8 + 3072X^2 + 16384
\]
at $p = 2$. We start with initial precision $s = 103$, for which we have the initial factorisation into the factors
\[
 \begin{array}{rcl}
  g_{(1)}(X)  & =  &  X + 4806835024200164988203597724980             \\
  g_{(2)}(X)  & =  &  X - 4806835024200164988203597724980             \\
  g_{(3)}(X)  & =  &  X^6 - 1093062124198142780466248559984X^4        \\
              &    &  - 4943636030726675686411786481408X^2            \\
              &   & - 4341143474460317541052331090944    .            \\
 \end{array}
\]

&

We obtain the following results in the first 10 steps. The defect is bounded above by $t = 23$.
Since $f(X)\equiv_2 X^8$, the defect is even bounded above by $t' = 22$.

\begin{center}
\begin{footnotesize}
\begin{tabular}{|l|c|c|c|c|} \hline
 step   & current      & defect \vsp{-1}     \\ 
        & precision    &        \vsp{-1}     \\
        & $s$          & $s - s'$            \\
                                             \hline
 1      & 103          & 3                   \\
                                             \hline
 2      & 200          & 4                   \\
                                             \hline
 3      & 392          & 5                   \\
                                             \hline
 4      & 774          & 1                   \\
                                             \hline
 5      & 1546         & 9                   \\
                                             \hline
 6      & 3074         & 3                   \\
                                             \hline
 7      & 6142         & 7                   \\
                                             \hline
 8      & 12270        & 3                   \\
                                             \hline
 9      & 24534        & 7                   \\
                                             \hline
10      & 49054        & 3                   \\
                                             \hline
\end{tabular}
\end{footnotesize}
\end{center}

\end{tabular}
\end{Example}

\begin{Example}
\label{Ex_u41}\rm\indent\vsp{-0.2}

\begin{tabular}{p{10.5cm}p{6cm}}

We consider the polynomial 
\[
f(X) =  X^{10} + 54X - 243
\]
at $p = 3$.
We start with initial precision $s = 46$, for which we have the initial factorisation into the factors
\[
 \begin{array}{rcl}
  g_{(1)}(X)  & =  & X + 1254845291302170687078       \\
  g_{(2)}(X)  & =  & X^3 + 3439114880299728595329X^2  \\
              &    & + 2097912255269159518284X        \\ 
              &    & + 2387878303991212496958         \\
  g_{(3)}(X)  & =  &  X^6 + 4168977948050601813522X^5 \\ 
              &    & + 3414335924445189447372X^4      \\
              &    &  - 469523799801953629710X^3      \\
              &    & - 3733781694469525960542X^2      \\
              &    &  + 2741122263554615006433X       \\
              &    & + 3057293995913895085035  .      \\
 \end{array}
\]

&

We obtain the following results in the first 10 steps. The defect is bounded above by $t = 13$.
Since $f(X)\equiv_3 X^{10}$, the defect is even bounded above by $t' = 10$.

\begin{center}
\begin{footnotesize}
\begin{tabular}{|l|c|c|c|c|} \hline
 step & current   & defect \vsp{-1}    \\ 
      & precision &        \vsp{-1}    \\
      & $s$       & $s - s'$           \\
                                       \hline
 1    & 46        & 3                  \\
                                       \hline
 2    & 86        & 0                  \\
                                       \hline
 3    & 172       & 3                  \\
                                       \hline
 4    & 338       & 2                  \\
                                       \hline
 5    & 672       & 1                  \\
                                       \hline
 6    & 1342      & 2                  \\
                                       \hline
 7    & 2680      & 1                  \\
                                       \hline
 8    & 5358      & 2                  \\
                                       \hline
 9    & 10712     & 1                  \\
                                       \hline
10   & 21422      & 2                  \\
                                       \hline
\end{tabular}
\end{footnotesize}
\end{center}

The defect seems to be eventually periodic.

\end{tabular}
\end{Example}

\parskip0.0ex
\begin{footnotesize}

\end{footnotesize}
\parskip1.2ex

\begin{footnotesize}
\begin{flushright}
Juliane Dei\ss ler\\
University of Stuttgart\\
Fachbereich Mathematik\\
Pfaffenwaldring 57\\
D-70569 Stuttgart\\
juliane.deissler@mint-kolleg.de\\
\end{flushright}
\end{footnotesize}


\begin{thebibliography}{99}
\bibitem{Magma} {\sc Bosma, W.; Cannon, J.J.; Fieker, C.; Steel, A.} (eds.), {\it Handbook of Magma functions,} Edition~2.16, 2010; 
               cf.\ magma.maths.usyd.edu.au.
\bibitem{Co82}  {\sc Cohn, P.M.,} {\it Algebra,} Vol.\ 1, Wiley, 1982.
\bibitem{De30}  {\sc Dedekind, R.} {\it Gesammelte Mathematische Werke,} Band 1, Vieweg, 1930.
\bibitem{FPR02} {\sc Ford, D.; Pauli, S.; Roblot, X.-F.,} {\it A fast algorithm for polynomial factorization over $\Q_p\,$,} J.\ Th.\ Nombres Bordeaux~14~(1), p.\ 151--169, 2002. 
\bibitem{Fr07} {\sc Frei, G.,} {\it The Unpublished Section Eight\,: On the Way to Function Fields over a Finite Field,} 
in\,: The Shaping of Arithmetic after C.\ F.\ Gauss's Disquisitiones Arithmeticae, Springer, 2007.
\bibitem{Ga76}  {\sc Gau\ss, C.F.,} {\it Werke,} Band II, zweiter Abdruck, 1876.
\bibitem{He04}  {\sc Hensel, K.,} {\it Neue Grundlagen der Arithmetik,} J.\ Reine Angew.\ Math.~127, p.\ 51--84, 1904.
\bibitem{Ko97}  {\sc Koch, H.,} {\it Zahlentheorie,} Vieweg, 1997.
\bibitem{vdW60} {\sc van der Waerden, B. L.,} {\it Algebra,} Springer Grundlehren, 5.\ Aufl., 1960.
\end{thebibliography}
\end{document}